\documentclass[11pt,twoside]{article}
\usepackage{amssymb,amsmath,amsthm,amsfonts,mathrsfs,hyperref}
\usepackage{times}
\usepackage{enumerate}
\usepackage{cite,titletoc}
\usepackage[toc,page,title,titletoc,header]{appendix}

\allowdisplaybreaks

\def\cF{\mathcal F}

\def\cJ{\mathcal J}

\def\cP{\mathcal P}
\def\cQ{\mathcal Q}
\def\cR{\mathcal R}

\def\cX{\mathcal X}

\def\N{\mathop{\mathbb N\kern 0pt}\nolimits}
\def\Z{\mathop{\mathbb Z\kern 0pt}\nolimits}
\def\Q{\mathop{\mathbb Q\kern 0pt}\nolimits}
\def\R{\mathop{\mathbb R\kern 0pt}\nolimits}
\def\T{\mathop{\mathbb T\kern 0pt}\nolimits}
\def\SS{\mathop{\mathbb S\kern 0pt}\nolimits}

\def\ds{\displaystyle}
\def\f{\frac}

\def\supp{\mathop{\rm supp}\nolimits}

\def\p{\partial}

\def\ve{\varepsilon}
\def\vp{\varphi}

\def\supp{\operatorname{supp}}
\def\ls{\lesssim}

\newcommand{\w}[1]{\langle {#1} \rangle}

\hypersetup{colorlinks=true,linkcolor=blue,citecolor=red,urlcolor=cyan}

\theoremstyle{plain}
\newtheorem{theorem}{Theorem}[section]

\newtheorem{lemma}[theorem]{Lemma}
\newtheorem{corollary}[theorem]{Corollary}

\theoremstyle{definition}
\newtheorem{definition}{Definition}[section]

\newtheorem{remark}{Remark}[section]

\numberwithin{equation}{section}

\title{The partial null conditions and global smooth solutions of the nonlinear wave equations on $\mathbb{R}^d\times\mathbb{T}$ with $d=2,3$}

\author{Fei Hou$^{1,*}$ \quad Fei Tao$^{1,*}$ \quad Huicheng Yin$^{2, }
    $\footnote{Fei Hou (\texttt{fhou$@$nju.edu.cn}), Fei Tao (\texttt{dg1721009@smail.nju.edu.cn}) and Huicheng Yin (\texttt{huicheng$@$nju.edu.cn}, \texttt{05407$@$njnu.edu.cn}) are supported by the NSFC (No.~11731007, No.~12101304). In addition, Fei Hou is  supported by the NSF of Jiangsu Province (No.~BK20210170), Yin Huicheng is supported by the National key research and development program of China (No.2020YFA0713803).}\\
    [12pt]{\small 1. Department of Mathematics, Nanjing University, Nanjing, 210093, China}\\
    {\small 2. School of Mathematical Sciences and Mathematical Institute,}\\
    {\small Nanjing Normal University, Nanjing, 210023, China}}

\begin{document}

\date{}
\maketitle
\thispagestyle{empty}

\begin{abstract} In this paper, we investigate the fully nonlinear wave equations
on the product space $\mathbb{R}^3\times\mathbb{T}$ with quadratic nonlinearities and
on $\mathbb{R}^2\times\mathbb{T}$ with cubic nonlinearities, respectively. It is shown that
for the small initial data satisfying some space-decay rates at infinity, these nonlinear equations admit global
smooth solutions when the corresponding partial null conditions hold and while have almost global smooth solutions
when the partial null conditions are violated. Our proof relies on the Fourier mode decomposition of the solutions
with respect to the periodic direction, the efficient combinations of time-decay estimates for the solutions
to the linear wave equations and the linear Klein-Gordon equations, and the global weighted energy estimates.
In addition, an interesting  auxiliary energy is introduced. As a byproduct, our results can be applied
to the 4D irrotational compressible Euler equations of polytropic gases or Chaplygin gases on $\mathbb{R}^3\times\mathbb{T}$,
the 3D relativistic membrane equation and the 3D nonlinear membrane equation on $\mathbb{R}^2\times\mathbb{T}$.

\vskip 0.2 true cm


\qquad \quad null condition, auxiliary energy, weighted energy estimate

\vskip 0.2 true cm
\noindent
\end{abstract}

\vskip 0.2 true cm

\addtocontents{toc}{\protect\thispagestyle{empty}}
\tableofcontents

\section{Introduction}
In this paper, we are concerned with the Cauchy problem of the following fully nonlinear wave equation
with quadratic nonlinearities on the product space $\R^3\times\T$
\begin{equation}\label{QWE}
\left\{
\begin{aligned}
&\Box u=Q(\p u,\p^2u),\qquad(t,x,y)\in(0,\infty)\times\R^3\times\T,\\
&(u,\p_tu)(0,x,y)=(u_{(0)},u_{(1)})(x,y),
\end{aligned}
\right.
\end{equation}
where $x=(x_1,x_2,x_3)\in\R^3$, $y\in\T:=\R/(2\pi\Z)$, $\Box:=\p_t^2-\Delta_{x}-\p_y^2$,
$\p:=(\p_0,\p_1,\cdots,\p_4)$, $\p_0:=\p_t$, $\p_i:=\p_{x_i}$ for $i=1,2,3$, $\p_4:=\p_y$,
$\Delta_x:=\sum_{i=1}^3\p_i^2$.
Note that for the small data problem, the cubic and higher order nonlinearities in $Q(\p u,\p^2u)$
can be neglected. Without loss of generality, one can assume that
\begin{equation}\label{nonlinear}
Q(\p u,\p^2u)=\sum_{\alpha,\beta,\gamma,\delta=0}^4
F^{\alpha\beta\gamma\delta}\p^2_{\alpha\beta}u\p^2_{\gamma\delta}u
+\sum_{\alpha,\beta,\gamma=0}^4Q^{\alpha\beta\gamma}\p^2_{\alpha\beta}u\p_\gamma u
+\sum_{\alpha,\beta=0}^4S^{\alpha\beta}\p_\alpha u\p_\beta u,
\end{equation}
where $F^{\alpha\beta\gamma\delta},Q^{\alpha\beta\gamma},S^{\alpha\beta}$ are constants with the symmetric conditions $F^{\alpha\beta\gamma\delta}=F^{\beta\alpha\gamma\delta}=F^{\alpha\beta\delta\gamma}=F^{\gamma\delta\alpha\beta}$, $Q^{\alpha\beta\gamma}=Q^{\beta\alpha\gamma}$
and $S^{\alpha\beta}=S^{\beta\alpha}$.

Meanwhile, we also investigate the Cauchy problem of the following  fully nonlinear wave equations
with cubic nonlinearities on the product space $\R^2\times\T$
\begin{equation}\label{QWE-cubic}
\left\{
\begin{aligned}
&\Box u=\mathcal{C}(\p u,\p^2 u),\qquad(t,x,y)\in(0,\infty)\times\R^2\times\T,\\
&(u,\p_tu)(0,x,y)=(u_{(0)},u_{(1)})(x,y),
\end{aligned}
\right.
\end{equation}
where $x=(x_1,x_2)\in\R^2$, $y\in\T:=\R/(2\pi\Z)$, $\Box:=\p_t^2-\Delta_{x}-\p_y^2$,
$\p:=(\p_0,\p_1,\p_2,\p_3)=(\p_t, \p_{x_1}, \p_{x_2},\p_{y})$, $\Delta_x:=\sum_{i=1}^2\p_i^2$.
As in \eqref{nonlinear}, without loss of generality, the cubic nonlinearities $\mathcal{C}(\p u,\p^2u)$ can be written as:
\begin{equation}\label{nonlinear-cubic}
\begin{split}
\mathcal{C}(\p u,\p^2u)&=\sum_{\alpha,\beta,\gamma,\delta,\mu,\nu=0}^3
F^{\alpha\beta\gamma\delta\mu\nu}\p^2_{\alpha\beta}u\p^2_{\gamma\delta}u\p^2_{\mu\nu}u
+\sum_{\alpha,\beta,\gamma,\delta,\mu=0}^3 G^{\alpha\beta\gamma\delta\mu}\p^2_{\alpha\beta}u\p^2_{\gamma\delta}u\p_{\mu}u\\
&+\sum_{\alpha,\beta,\gamma,\delta=0}^3 H^{\alpha\beta\gamma\delta}\p^2_{\alpha\beta}u\p_\gamma u\p_\delta u+
\sum_{\alpha,\beta,\gamma=0}^3S^{\alpha\beta\gamma}\p_{\alpha}u\p_\beta u\p_\gamma u,
\end{split}
\end{equation}
where $F^{\alpha\beta\gamma\delta\mu\nu},G^{\alpha\beta\gamma\delta\mu},H^{\alpha\beta\gamma\delta},S^{\alpha\beta\gamma}$ are constants
with the symmetric conditions $F^{\alpha\beta\gamma\delta\mu\nu}=F^{\gamma\delta\alpha\beta\mu\nu}$ $=F^{\mu\nu\gamma\delta\alpha\beta}=F^{\alpha\beta\mu\nu\gamma\delta}
=F^{\beta\alpha\gamma\delta\mu\nu}=F^{\alpha\beta\delta\gamma\mu\nu}=F^{\alpha\beta\delta\gamma\nu\mu}$, $G^{\alpha\beta\gamma\delta\mu}=G^{\gamma\delta\alpha\beta\mu}=G^{\beta\alpha\gamma\delta\mu}=G^{\alpha\beta\delta\gamma\mu}$, $H^{\alpha\beta\gamma\delta}=H^{\beta\alpha\gamma\delta}=H^{\alpha\beta\delta\gamma}$ and $S^{\alpha\beta\gamma}=S^{\beta\alpha\gamma}=S^{\alpha\gamma\beta}$.

\begin{definition}\label{YH-1}[Partial null condition for \eqref{QWE} with \eqref{nonlinear}]
If
\begin{equation}\label{null:def}
\sum_{\alpha,\beta,\gamma,\delta=0}^3F^{\alpha\beta\gamma\delta}\xi_\alpha\xi_\beta\xi_\gamma\xi_\delta\equiv0,
\quad\sum_{\alpha,\beta,\gamma=0}^3Q^{\alpha\beta\gamma}\xi_\alpha\xi_\beta\xi_\gamma\equiv0,
\quad\sum_{\alpha,\beta=0}^3S^{\alpha\beta}\xi_\alpha\xi_\beta\equiv0
\end{equation}
hold for any $\xi=(\xi_0,\xi_1,\xi_2,\xi_3)\in\{\pm1,\SS^2\}$, \emph{i.e.} $\xi_0^2=\xi_1^2+\xi_2^2+\xi_3^2=1$,
then we call \eqref{QWE} with \eqref{nonlinear} satisfies the partial null condition.
\end{definition}

Analogously, as in Definition \ref{YH-1}, one has
\begin{definition}\label{YH-2}[Partial null condition for \eqref{QWE-cubic} with \eqref{nonlinear-cubic}]
If
\begin{equation}\label{null:def-cubic}
\begin{split}
&\sum_{\alpha,\beta,\gamma,\delta,\mu,\nu=0}^2 F^{\alpha\beta\gamma\delta\mu\nu}\xi_\alpha\xi_\beta\xi_\gamma\xi_\delta\xi_\mu\xi_\nu\equiv0,
\quad\sum_{\alpha,\beta,\gamma,\delta,\mu=0}^2 G^{\alpha\beta\gamma\delta\mu}\xi_\alpha\xi_\beta\xi_\gamma\xi_\delta\xi_\mu\equiv0,\\
&\sum_{\alpha,\beta,\gamma,\delta=0}^2 H^{\alpha\beta\gamma\delta}\xi_\alpha\xi_\beta\xi_\gamma\xi_\delta \equiv0,
\quad\sum_{\alpha,\beta,\gamma=0}^2 S^{\alpha\beta\gamma}\xi_\alpha\xi_\beta\xi_\gamma \equiv0
\end{split}
\end{equation}
hold for any $\xi=(\xi_0,\xi_1,\xi_2)\in\{\pm1,\SS^1\}$, i.e. $\xi_0^2=\xi_1^2+\xi_2^2=1$,
then we call \eqref{QWE-cubic} with \eqref{nonlinear-cubic} satisfies the partial  null condition.
\end{definition}

It is pointed out that there are no any restrictions on the $y-$directon in the partial null conditions of Definitions \ref{YH-1}
and \ref{YH-2}.

Set
\begin{equation}\label{initial:data}
\ve:=\sum_{i+j\le N+1}\|\w{x}^{2+i+j}\nabla_x^i\p_y^ju_{(0)}\|_{L^2_{x,y}}
+\sum_{i+j\le N}\|\w{x}^{3+i+j}\nabla_x^i\p_y^j u_{(1)}\|_{L^2_{x,y}},
\end{equation}
where $N\ge25$, $\nabla_x=(\p_1,..., \p_d)$, $\|\cdot\|_{L^2_{x,y}}:=\|\cdot\|_{L^2(\R^d\times\T)}$ with $d=2,3$
and $\w{x}:=\sqrt{1+|x|^2}$.

\vskip 0.3 true cm

The main results of this paper are:
\begin{theorem}\label{thm1}
For $\ve$ defined by \eqref{initial:data} with $d=3$, then

($A_1$) there is $\ve_0>0$ such that when $\ve\le\ve_0$, problem \eqref{QWE} with the partial null condition \eqref{null:def}
admits a global solution $u\in \bigcap\limits_{j=0}^{N+1}C^{j}([0,\infty), H^{N+1-j}(\R^3\times\T)))$.
Moreover, there is a constant $C>0$ such that
\begin{equation}\label{dyu:pw}
|\p u(t,x,y)|+|\p^2 u(t,x,y)|\le C\ve\w{t}^{-1},\quad |\p_y u(t,x,y)|\le C\ve\w{t+|x|}^{-3/2}.
\end{equation}

($A_2$) there are two positive constants $\ve_0$ and $\kappa_0$ such that when $\ve\le\ve_0$, problem \eqref{QWE} has a
solution $u\in \bigcap\limits_{j=0}^{N+1}C^{j}([0,T_\ve], H^{N+1-j}(\R^3\times\T)))$ with $T_\ve:=e^{\kappa_0/\ve}-1$.
In addition, \eqref{dyu:pw} holds for any $t\in[0,T_\ve]$.

\end{theorem}

\begin{theorem}\label{thm1-cubic}
For $\ve$ defined by \eqref{initial:data} with $d=2$, then

($B_1$) there is $\ve_0>0$ such that when $\ve\le\ve_0$, problem \eqref{QWE-cubic} with the partial null condition \eqref{null:def-cubic}
admits a global solution $u\in \bigcap\limits_{j=0}^{N+1}C^{j}([0,\infty), H^{N+1-j}(\R^2\times\T)))$. Furthermore, there is a constant $C>0$ such that
\begin{equation}\label{dyu:pw-cubic}
|\p u(t,x,y)|+|\p^2 u(t,x,y)|\le C\ve\w{t}^{-1/2},\quad |\p_y u(t,x,y)|\le C\ve\w{t+|x|}^{-1}.
\end{equation}

($B_2$) there are two positive constants $\ve_0$ and $\kappa_0$ such that when $\ve\le\ve_0$, problem \eqref{QWE-cubic}
has a solution $u\in \bigcap\limits_{j=0}^{N+1}C^{j}([0,T_\ve], H^{N+1-j}(\R^2\times\T)))$ with $T_\ve:=e^{\kappa_0/\ve^2}-1$.
In addition, \eqref{dyu:pw-cubic} holds for any $t\in[0,T_\ve]$.

\end{theorem}

\begin{remark}\label{HC1}
If the initial data $(u_{(0)}, u_{(1)})\in C_0^{\infty}(\Bbb R^4)$ in \eqref{QWE}
with $\|u_{(0)}\|_{H^7(\Bbb R^4)}+\|u_{(1)}\|_{H^6(\Bbb R^4)}$ being small
or $(u_{(0)}, u_{(1)})\in C_0^{\infty}(\Bbb R^3)$ in \eqref{QWE-cubic}
with $\|u_{(0)}\|_{H^5(\Bbb R^3)}+\|u_{(1)}\|_{H^4(\Bbb R^3)}$ being small,
then it follows from the results in \cite{K-P} or Chapter 6 of \cite{Hormander97book} that both \eqref{QWE}
and \eqref{QWE-cubic} have global smooth solutions. In addition, if the initial data in \eqref{QWE} and \eqref{QWE-cubic}
are independent of the variable $y$, which is certainly periodic with respect to $y$, then the equations
in $(A_1)$ of Theorem \ref{thm1} and $(B_1)$ of Theorem \ref{thm1-cubic}
actually become the 3D and 2D fully nonlinear wave equations with the null condition structures, respectively.
In this case, the corresponding 3D and 2D nonlinear wave equations have global smooth small
solutions (see \cite{Christodoulou86}, \cite{Klainerman} and \cite{Alinhac01b}).
This motivates us to obtain the global classical solutions
to the more general problems in $(A_1)$ of Theorem \ref{thm1} on $\R^3\times\T$
and $(B_1)$ of Theorem \ref{thm1-cubic} on $\R^2\times\T$.
\end{remark}

\begin{remark}\label{HC2}
Note that the results of almost global
solutions in $(A_2)$ of Theorem \ref{thm1} and $(B_2)$ of Theorem \ref{thm1-cubic} are optimal.
Indeed, for examples, if the initial data in \eqref{QWE} and \eqref{QWE-cubic} are independent of the variable $y$,
then the equations in \eqref{QWE} and \eqref{QWE-cubic}
actually become the 3D and 2D nonlinear wave equations without the null condition structures, respectively.
Therefore, by the optimal orders of $\ve$ for the lifespans of smooth solutions in \cite{Alinhac01a} and \cite{Alinhac01b},
we know that the conclusions in $(A_2)$ of Theorem \ref{thm1} and $(B_2)$ of Theorem \ref{thm1-cubic} for $T_{\ve}$
may not be improved.
\end{remark}

\begin{remark}\label{HC3}
The 4D compressible isentropic Euler equations are
\begin{equation}\label{Euler}
\left\{
\begin{aligned}
&\p_t\rho+div (\rho v)=0,\\
&\p_t(\rho v)+div (\rho v \otimes v)+\nabla p=0,\\
\end{aligned}
\right.
\end{equation}
which express the conservations of the mass and momentum, respectively. Where $(t,x, y)\in [0,+\infty)\times\mathbb{R}^3\times\Bbb T$, $\nabla=(\p_1,\p_2, \p_3,\p_4)=(\p_{x_1},\p_{x_2}, \p_{x_3},
\p_y)$, $v=(v_1,v_2,v_3,v_4)$, $\rho$, $p$ stand for the velocity, density, pressure
respectively. For the polytropic gases, the equation of state (see \cite{CF}) is given by
\begin{equation}\label{pr_0}
p(\rho)=A\rho^{\gamma},
\end{equation}
where $A$ and $1<\gamma<3$ are positive constants.

For the Chaplygin gases, the equation of state (see \cite{CF}) is
\begin{equation}\label{pr}
p(\rho)=P_0-\frac{B}{\rho}
\end{equation}
with $P_0$ and $B$ being positive constants.

Suppose that \eqref{Euler} admits the irrotational smooth initial data with $2\pi$-period on $y$:
$$
\big(\rho,v)(0,x,y)=(\bar\rho+\rho_0(x,y),v_1^0(x,y),v_2^0(x,y),v_3^0(x,y),v_4^0(x,y)\big),
$$
where $\bar\rho+\rho_0(x,y)>0$, and $\bar\rho$ is
a positive constant which can be normalized so that the sound speed $c(\bar\rho)=\sqrt{p'(\bar\rho)}=1$. Then by the irrotationality
of \eqref{Euler}, one can
introduce a potential function $\phi$ such that $v=\nabla\phi$, and the Bernoulli's law, $\p_t\phi+\frac{1}{2}|\nabla\phi|^2+h(\rho)=0$,
holds with the enthalpy $h(\rho)$ satisfying $h'(\rho)=\f{c^2(\rho)}\rho$ and $h(\bar\rho)=0$. Therefore, for smooth
irrotational flows, \eqref{Euler} is equivalent to that for the polytropic gases,
\begin{equation}\label{1.9-00}
\begin{split}
\p_t^2\phi-c^2(\rho)\Delta \phi+2\ds\sum_{k=1}^4\p_k\phi\p_{tk}^2\phi+\ds\sum_{i,j=1}^4\p_i\phi\p_j\phi\p_{ij}^2\phi
=0;
\end{split}
\end{equation}
for Chaplygin gases,
\begin{equation}\label{1.9-01}
\begin{split}
\p_t^2\phi-(1+2\p_t\phi)\triangle\phi+2\sum_{k=1}^4\p_k\phi\p_{tk}\phi
+\sum_{i,j=1}^4\p_i\phi\p_j\phi\p_{ij}^2\phi-|\nabla\phi|^2\triangle\phi=0.
\end{split}
\end{equation}
It is easy to check that \eqref{1.9-01}
satisfies the partial null condition in \eqref{null:def} while \eqref{1.9-00} does not.
Therefore, Theorem \ref{thm1} holds true for the 4D irrotationlal compressible Euler equations \eqref{Euler}
of Chaplygin gases and polytropic gases on $\R^3\times\T$.
\end{remark}

\begin{remark}\label{HC4}
The 3D nonlinear relativistic membrane equation is
\begin{equation}\label{HCC-1}
\p_t\Big(\ds\f{\p_t\phi}{\sqrt{1-(\p_t\phi)^2+|\nabla\phi|^2}}\Big)
-\ds\sum_{i=1}^3\p_i\Big(\ds\f{\p_i\phi}{\sqrt{1-(\p_t\phi)^2+|\nabla\phi|^2}}\Big)=0,
\end{equation}
where $(t,x, y)\in [0,+\infty)\times\mathbb{R}^2\times\Bbb T$, $\nabla=(\p_1,\p_2, \p_3)=(\p_{x_1},\p_{x_2}, \p_y)$.
For small smooth solution $\phi$, \eqref{HCC-1} is equivalent to
\begin{equation}\label{HCC-2}
\p_t^2\phi-\Delta\phi+|\nabla\phi|^2\p_t^2\phi+\Delta\phi((\p_t\phi)^2-|\nabla\phi|^2)+\ds\sum_{k=1}^3\p_k\phi
\nabla\p_k\phi\cdot\nabla\phi-2\p_t\phi\nabla\phi\cdot\nabla\p_t\phi=0.
\end{equation}
It follows from direct computation that \eqref{HCC-2} fulfills the partial null condition in Definition \ref{YH-2}.
Hence,  $(B_1)$ of Theorem \ref{thm1-cubic} can be applied to equation \eqref{HCC-1} with the initial data in \eqref{QWE-cubic}
on $\Bbb R^2\times\Bbb T$.

On the other hand, consider the following 3D nonlinear wave equation which can be regarded as a model equation for the nonlinear
version of Maxwell equations in nonlinear electromagnetic theory (see \cite{MY})
\begin{equation}\label{HCC-3}
-(1+3G''(0)(\p_t \phi)^2)\p_t^2 \phi+\Delta \phi=0,
\end{equation}
where $(t,x, y)\in [0,+\infty)\times\mathbb{R}^2\times\Bbb T$, and the constant $G''(0)\not=0$.
It is easy to know that \eqref{HCC-3} does not satisfy the partial null condition in Definition \ref{YH-2}.
Therefore, the result in $(B_2)$ of Theorem \ref{thm1-cubic} holds for \eqref{HCC-3} on $\Bbb R^2\times\Bbb T$.
Analogously, Theorem \ref{thm1-cubic} $(B_2)$ is also true for the 3D nonlinear membrane equation on $\Bbb R^2\times\Bbb T$
$$\p_t^2\phi
-\ds\sum_{i=1}^3\p_i\Big(\ds\f{\p_i\phi}{\sqrt{1+|\nabla\phi|^2}}\Big)=0.$$

\end{remark}

\vskip 0.1 true cm

The topic on the well-posedness of wave equations on product spaces arises from
the studies for the propagation of waves along infinite homogeneous waveguides (see \cite{Ett-15,PLR-03,MSS-05})
and the Kaluza-Klein theory (see \cite{Kalu}, \cite{Klein}, \cite{Witten}). It also comes from the investigation of the
global stability problem for Einstein equations with supersymmetric compactifications. In fact, by the wave-coordinates,
a system of quasilinear wave equations on the product space $\R^{d}\times\T$ can be derived as a toy model for Einstein equations
with additional compact dimension. Its natural feature is of the spatial anisotropy, which is different from the
classical theory for the hyperbolic equations on $\mathbb{R}^d$. We refer the readers to the
more introductions in \cite{HS-21} for the detailed backgrounds of wave equations on product spaces.

We point out that the study of the wave equations on the product space $\R^{d}\times\T$ is closely
related to the Klein-Gordon system in $\R^{d}$. Indeed, if taking the Fourier transformation with respect to
the variable $y$ for the nonlinear equation \eqref{QWE}, then we obtain infinite many coupled nonlinear equations in $\R^3$
as follows
\begin{equation}\label{reduction to WKG}
\begin{split}
&(\p_t^2-\Delta_x+|n|^2)u_n(t,x)=(Q(\p u,\p^2u))_n(t,x),\quad x\in \R^3, n\in\Z,
\end{split}
\end{equation}
where
\begin{equation}\label{reduction to WKG-00}
\begin{split}
&u_n(t,x):=\frac{1}{2\pi}\int_{\T}e^{-{\rm i}ny}u(t,x,y)dy.\\
\end{split}
\end{equation}
From \eqref{reduction to WKG}, one knows that the zero-mode $u_0$ is the solution of a wave equation
and other non-zero mode $u_n$ ($n\not=0$) solves a Klein-Gordon equation with mass $|n|$.
Note that the system \eqref{reduction to WKG} is composed of a wave equation and an infinite number of coupled
Klein-Gordon equations with variable masses,
which leads to essential difficulties to solve them directly. Let's review some results related to a finite number of coupling
systems of  the wave equation and Klein-Gordon equations in $\R^{1+d} (d=2,3)$ with small data. In \cite{Georgiev-90},
V.Georgiev introduced the strong null condition of the nonlinearity with the forms $\p_\alpha u\p _\beta v-\p_\alpha v\p _\beta u$
($\alpha,\beta\in \{0,1,2,3\}$) and further obtained the global well-posedness of small data solutions for
the coupled system of wave and Klen-Gordon equations in $\R^{1+3}$. Subsequently, S.Katayama \cite{Katayama12}
established the global existence under a weaker condition than the strong null condition imposed in \cite{Georgiev-90}.
Recently, the following wave-Klein-Gordon system
in $\R^{1+3}$ was studied by some authors \cite{Q-Wang-2015,IP19,LeF-Ma-16,Q-Wang-2020}:
\begin{equation}\label{Simplied-WKG-eq}
\begin{cases}
&-\Box u=A^{\alpha\beta}\p_\alpha v \p_\beta v+Dv^2,\\
&(-\Box +1)v=u B^{\alpha\beta}\p_\alpha\p_\beta v.
\end{cases}
\end{equation}
The system \eqref{Simplied-WKG-eq} was derived in \cite{LeF-Ma-16} and \cite{Q-Wang-2015}
as a simplified model for the full Einsten-Klein-Gordon system. In addition,
 the authors in \cite{LeF-Ma-16} proved the existence of global-in-time solutions to the Cauchy problem of
 \eqref{Simplied-WKG-eq} for sufficiently smooth and compactly supported initial data.
For the small smooth initial data with suitable spatial decay at infinity (not necessarily compactly supported),
the authors in \cite{IP19} established the global regularity and modified scattering of \eqref{Simplied-WKG-eq}.
Concerning the global well-posedness of small amplitude solutions of the coupled wave and Klein-Gordon equations in lower space dimensions,
we refer the readers to see \cite{Dong-JFA}-\cite{Yue-M-6}, \cite{Ifrim-Stingo} and \cite{Stingo-Memoirs}.

In the paper, motivated by the works above, we investigate the long time existence of
nonlinear wave equations on the product space $\R^{d}\times\T$ with $d=2,3$.
As far as we know, there are few results in this direction. Recently, the authors in \cite{HS-21} considered a
system of quasilinear wave equations on the product space $\R^{3}\times\T$ as follows
\begin{equation}\label{HCC-4}
\begin{cases}
&-\p_t^2u+\Delta_xu+(1+u)\p_y^2u=\ds\sum_{1\le i,j\le 2}N_1(w_i,w_j),\\
&-\p_t^2v+\Delta_xv+(1+u)\p_y^2v=\ds\sum_{1\le i,j\le 2}N_2(w_i,w_j),\\
\end{cases}
\end{equation}
where $(t,x,y)\in [0,\infty)\times\R^3\times\T$, $w_1,w_2=\{u,v\}$,  the nonlinearities $N_1(\cdot,\cdot)$ and
$N_2(\cdot,\cdot)$ are linear combinations of the following quadratic null forms
\begin{equation*}
\begin{split}
&\p_t\phi\p_t\psi-\nabla_x\phi\cdot\nabla_x\psi,\quad \p_i\phi\p_j\psi-\p_i\psi\p_j\phi,\quad 1\le i<j\le 3,\\
&\p_t\phi\p_i\psi-\p_i\phi\p_t\psi,\quad 1\le i\le 3
\end{split}
\end{equation*}
with $\phi,\psi=\{u,v\}$. For the small and regular initial data of \eqref{HCC-4} with polynomial spatial decay at infinity,
the global existence of small data solution $(u,v)$ is shown.
Compared with the result in \cite{HS-21}, we establish the global or almost global solutions of 4D fully nonlinear wave equations
on $\R^{3}\times\T$ when the general partial null condition is fulfilled or not. In addition, we also establish the global or almost global
small data solutions of 3D fully nonlinear wave equations on $\R^{2}\times\T$ with cubic nonlinearities. It is pointed out that
the methods in \cite{HS-21} and our present paper are different.
On the other hand, the interested readers are also referred to the works \cite{LTY22a,TY21} on the long time
behavior of Klein-Gordon equations on the product space $\R^{d}\times\T$ with $d=2,3$.

We now give some comments on the proofs of Theorem \ref{thm1} and Theorem \ref{thm1-cubic}.
Note that the following Klainerman-Sobolev inequality plays a key role
in proving the global existence of small data smooth solutions to the $d$-dimensional nonlinear wave equations
(see \cite{Alinhac01a,Alinhac01b,Christodoulou86,Hormander97book, Klainerman})
\begin{equation}\label{HCC-5}
|\p \phi(t,x)|\le \f{C}{(1+|t-r|)^{\f12}(1+t)^{\f{d-1}{2}}}\ds\sum_{|I|\le [\f{d}{2}]+2}\|\hat{Z}^I\p \phi(t,x)\|_{L_x^2(\Bbb R^d)},
\end{equation}
where $\hat{Z}\in\{\p, x_i\p_j-x_j\p_i, x_i\p_t+t\p_i, 1\le i,j\le d, t\partial_t+\sum_{k=1}^d x_k\p_k\}$ and $r=|x|=\sqrt{x_1^2+\cdot\cdot\cdot+x_d^2}$.
However, due to the spatial anisotropy of the product space $\Bbb R^d\times\Bbb T$, the scaling vector $S=t\partial_t+\sum_{i=1}^{d}x_i\partial_{x_i}$ does not commutate with the wave operator $\Box:=\p_t^2-\Delta_{x}-\p_y^2$ and
yet does not commutate with the resulted Klein-Gordon operators $\p_t^2-\Delta_x+|n|^2$ ($n\in \Z$) appeared in \eqref{reduction to WKG}.
This fact prevents us from using the classical Klainerman-Sobolev inequality \eqref{HCC-5} to study the global solution problems
of \eqref{QWE} and \eqref{QWE-cubic}.
To overcome this difficulty, our key ingredients are to take the pointwise estimates for the solutions of  \eqref{QWE} and \eqref{QWE-cubic}
inspired the vector fields method in \cite{Klainerman85,Klainerman}.
Next we only give the detailed explanations on the pointwise estimates of solutions
to problem \eqref{QWE} since the treatments on \eqref{QWE-cubic} are analogous.
Define the hyperbolic rotations $L_i:=x_i\p_t+t\p_i,i=1,2,3$, $L:=(L_1,L_2,L_3)$ and the space rotations $\Omega_{ij}:=x_i\p_j-x_j\p_i,i,j=1,2,3$, $\Omega:=(\Omega_{23},\Omega_{31},\Omega_{12})$. Set $\p_{t,x}:=(\p_0,\p_1,\p_2,\p_3)$ and $Z:=\{Z_1,\cdots,Z_{11}\}=\{\p_{t,x},L,\Omega,\p_y\}$.
The procedure on the pointwise estimates of $Z^au$ ($a\in\Bbb N_0^{11}$) will be divided into two parts to
treat the non-zero modes $n\neq0$ and the zero-mode $n=0$ respectively.
For notational convenience, the following projections of  $f(t,x,y)$ are defined:
\begin{equation}\label{proj:def}
P_{=0}f(t,x,y):=\frac{1}{2\pi}\int_{\T}f(t,x,y')dy'=f_0(t,x),\quad P_{\neq0}:={\rm Id}-P_{=0}.
\end{equation}
With respect to the non-zero mode $(Z^au)_n(t,x)$, by virtue of the $L^\infty$-$L^2$ estimates on the Klein-Gordon
equation \eqref{reduction to WKG} with variable mass (see \cite{Georgiev92}), we can obtain
\begin{equation}\label{ansatz:nonzero}
\w{t+|x|}^{1.4}|(Z^au)_n(t,x)|\le C|n|^{|a|+7-N}\ve,\quad|a|\le N-8,n\in\Z_*:=\Z\setminus\{0\}.
\end{equation}
With respect to the zero mode $(P_{=0} Z^au)(t,x)$, we will introduce the following auxiliary energy
to estimate it
\begin{equation}\label{aux:energy:def}
\cX_k(t):=\sum_{|a|\le k-1}\Big(\|\w{t-|x|}P_{=0}\p^2Z^au\|_{L^2(\R^3)}
+\|\w{t+|x|}P_{=0}\bar\p\p Z^au\|_{L^2(\R^3:|x|\ge \w{t}/3)}\Big),
\end{equation}
where $\bar\p=(\bar\p_1,\bar\p_2,\bar\p_3)$ and $\bar\p_i:=\p_i+\frac{x_i}{|x|}\p_t$, $i=1,2,3$.
The introduction of $\cX_k(t)$ is due to such two reasons: the weight $\w{t-|x|}$ in the first term of \eqref{aux:energy:def}
will provide the required temporal decay rate away from the outgoing light conic surface $|x|=t$;
the appearances of the weight $\w{t+|x|}$ and the good derivative $\bar\p$ in the second term of \eqref{aux:energy:def}
will produce a better time-decay rate $\w{t}^{-3/2}$ of
$(P_{=0} Z^au)(t,x)$ near the conic surface $|x|=t$ when the null condition of related nonlinearity is fulfilled
or produce at least the time-decay of $\w{t}^{-1}$ for $(P_{=0} Z^au)(t,x)$ when the null condition is violated.
Based on this together with \eqref{ansatz:nonzero}, Theorem \ref{thm1} can be proved by the continuous argument.
Analogously, the proof of Theorem \ref{thm1-cubic} can be completed.

\vskip 0.1 true cm

Notations:
\begin{itemize}
  \item $\w{x}:=\sqrt{1+|x|^2}$.
  \item $\N_0:=\{0,1,2,\cdots\}$, $\Z:=\{0,\pm1,\pm2,\cdots\}$ and $\Z_*:=\Z\setminus\{0\}$.
  \item On $\R^d\times\T$ with $d=2,3$, set $\p_0:=\p_t$, $\p_i:=\p_{x_i}$, $i=1,\cdot\cdot\cdot,d$,
  $\p_{d+1}:=\p_y$, $\p_x:=\nabla_x=(\p_1,\cdot\cdot\cdot,\p_d)$, $\p_{t,x}:=(\p_0,\p_1,\cdot\cdot\cdot,\p_d)$, $\p:=(\p_t,\p_1,\cdot\cdot\cdot,\p_d,\p_y)$ and $\Box_{t,x}:=\p_t^2-\Delta_x=\p_t^2-\sum_{i=1}^{d}\p_i^2$.
  \item On $\R^d\times\T$ with $d=2,3$, for $|x|>0$, define $\bar\p_i:=\p_i+\frac{x_i}{|x|}\p_t$, $i=1,\cdot\cdot\cdot,d$ and $\bar\p=(\bar\p_1,\cdot\cdot\cdot,\bar\p_d)$.
  \item The following vector fields on $\R^d\times\T$ with $d=2,3$ are defined: $L_i:=x_i\p_t+t\p_i,i=1,\cdot\cdot\cdot,d$, $L:=(L_i)_{i=1,\cdot\cdot\cdot,d}$, $\Omega_{ij}:=x_i\p_j-x_j\p_i,1\leq i<j\leq d$, $\Omega:=(\Omega_{ij})_{1\leq i<j\leq d}$.

       On $\R^3\times\T$, $\Gamma=\{\Gamma_1,\cdot\cdot\cdot,\Gamma_{10}\}=:\{\p_{t,x},(L_i)_{1\leq i\leq3},(\Omega_{ij})_{1\leq i<j\leq 3}\}$ and $Z=\{\Gamma_1,\cdot\cdot\cdot,\Gamma_{11}\}=:\{\Gamma,\p_y\}$.

       On $\R^2\times\T$, $\Gamma=\{\Gamma_1,\cdots,\Gamma_{6}\}=:\{\p_{t,x},L_1,L_2,\Omega_{12}\}$ and $Z=\{Z_1,\cdots,Z_{7}\}=:\{\Gamma,\p_y\}$.
  \item
      On $\R^3\times\T$, $\Gamma^a:=\Gamma_1^{a_1}\Gamma_2^{a_2}\cdots\Gamma_{10}^{a_{10}}$ for $a\in\N_0^{10}$
      and $Z^a:=Z_1^{a_1}Z_2^{a_2}\cdots Z_{11}^{a_{11}}$
      for $a\in\N_0^{11}$. On $\R^2\times\T$, $\Gamma^a:=\Gamma_1^{a_1}\Gamma_2^{a_2}\cdots\Gamma_{6}^{a_{6}}$ for $a\in\N_0^{6}$
      and $Z^a:=Z_1^{a_1}Z_2^{a_2}\cdots Z_{7}^{a_{7}}$ for $a\in\N_0^7$.
  \item $\|\cdot\|_{L^p_x}:=\|\cdot\|_{L^p(\R^d)}$, $\|\cdot\|_{L^p_y}:=\|\cdot\|_{L^p(\T)}$ and $\|\cdot\|_{L^p_{x,y}}:=\|\cdot\|_{L^p(\R^d\times\T)}$ for $d=2,3$, respectively.
  \item Auxiliary $k$-order energy on $\R^d\times\T$ with $d=2,3$:
  \begin{equation*}
   \cX_k(t):=\sum_{|a|\le k-1}\Big(\|\w{t-|x|}P_{=0}\p^2Z^au\|_{L^2(\R^d)}+\|\w{t+|x|}P_{=0}\bar\p\p Z^au\|_{L^2(\R^d:|x|\ge \w{t}/3)}\Big),
  \end{equation*}
  \item $\ds f_n(t,x):=\frac{1}{2\pi}\int_{\T}e^{-{\rm i}ny}f(t,x,y)dy$, ${\rm i}:=\sqrt{-1}$, $n\in\Z$, $x\in \R^d$ with $d=2,3$.
  \item $A\ls B$ means $A\le CB$ for a generic constant $C>0$.
  \item For $\cP=\Omega$ or $\p$, denote $\ds|\cP^{\le j}f|:=\Big(\sum_{0\le|a|\le j}|\cP^af|^2\Big)^\frac12$ and $\ds|\cP^{\le j}f\cP^{\le k}g|:=\Big(\sum_{\substack{0\le|a|\le j,\\0\le|b|\le k}}|\cP^af\cP^bg|^2\Big)^\frac12$.
\end{itemize}

\section{Preliminaries and bootstrap assumptions}\label{sect2}
\subsection{Some basic lemmas}\label{sect2-1}

At first, we show some properties on the  orthogonality and commutation of the projection operators $P_{=0}$
or $P_{\neq0}$.
\begin{lemma}\label{lem:proj}
For any real valued functions $f(t,x,y)$ and $g(t,x,y)$, it holds that
\begin{equation*}
\begin{split}
&\int_{\T}P_{=0}fP_{\neq0}gdy=0,\quad
\|f\|^2_{L_y^2}=\|P_{=0}f\|^2_{L_y^2}+\|P_{\neq0}f\|^2_{L_y^2},\\
&P_{=0}Zf=ZP_{=0}f,\quad P_{\neq0}Zf=ZP_{\neq0}f,\quad Z=\{\p,L,\Omega\},\\
&P_{=0}\p_yf=\p_yP_{=0}f=0,\quad(\Gamma f)_n=\Gamma(f_n),\quad \Gamma=\{\p_{t,x},L,\Omega\}.\\
\end{split}
\end{equation*}
\end{lemma}
\begin{proof}
These properties can be directly verified, we omit the details here.
\end{proof}

The following two lemmas mean that the vector fields $Z^a$ commute with the wave operator $\Box$ and  the
corresponding partial null conditions are still preserved.
\begin{lemma}
Suppose that $u$ is a solution to \eqref{QWE}.
Then for any multi-index $a$, $Z^au$ satisfies
\begin{equation}\label{eqn:high}
\begin{split}
\Box Z^au&=Q^a:=Z^aQ(\p u,\p^2u)=\sum_{\alpha,\beta,\gamma,\delta=0}^4\sum_{b+c\le a}
F_{abc}^{\alpha\beta\gamma\delta}\p^2_{\alpha\beta}Z^bu\p^2_{\gamma\delta}Z^cu\\
&\qquad+\sum_{\alpha,\beta,\gamma=0}^4\sum_{b+c\le a}
Q_{abc}^{\alpha\beta\gamma}\p^2_{\alpha\beta}Z^bu\p_\gamma Z^cu
+\sum_{\alpha,\beta=0}^4\sum_{b+c\le a}S_{abc}^{\alpha\beta}\p_\alpha Z^bu\p_\beta Z^cu,
\end{split}
\end{equation}
where $F_{abc}^{\alpha\beta\gamma\delta}$, $Q_{abc}^{\alpha\beta\gamma}$ and $S_{abc}^{\alpha\beta}$ are constants,
 in particular, $F_{aa0}^{\alpha\beta\gamma\delta}=F_{a0a}^{\alpha\beta\gamma\delta}=F^{\alpha\beta\gamma\delta}$,  $Q_{aa0}^{\alpha\beta\gamma}=Q^{\alpha\beta\gamma}$.
Furthermore, if the partial null condition \eqref{null:def} holds, then for any $\xi=(\xi_0,\xi_1,\xi_2,\xi_3)\in\{\pm1,\SS^2\}$,
\begin{equation}\label{null:high}
\sum_{\alpha,\beta,\gamma,\delta=0}^3F_{abc}^{\alpha\beta\gamma\delta}\xi_\alpha\xi_\beta\xi_\gamma\xi_\delta\equiv0,
\quad\sum_{\alpha,\beta,\gamma=0}^3Q_{abc}^{\alpha\beta\gamma}\xi_\alpha\xi_\beta\xi_\gamma\equiv0,
\quad\sum_{\alpha,\beta=0}^3S_{abc}^{\alpha\beta}\xi_\alpha\xi_\beta\equiv0.
\end{equation}
\end{lemma}
\begin{proof}
\eqref{eqn:high} can be obtained by direct verification.
\eqref{null:high} comes from Lemma 6.6.5 of \cite{Hormander97book}, $\p_y\Gamma=\Gamma\p_y$ with $\Gamma=\{\p_{t,x},L,\Omega\}$
and direct computation.
\end{proof}

In the same way, we have:
\begin{lemma}
Suppose that $u$ is a solution to \eqref{QWE-cubic}.
Then $Z^au$ satisfies
\begin{equation}\label{eqn:high-cubic}
\begin{split}
\Box Z^au&=\mathcal{C}^a:=Z^a \mathcal{C}(\p u,\p^2u)=\sum_{\alpha,\beta,\gamma,\delta,\mu,\nu=0}^3\sum_{b+c+d\le a}
F^{\alpha\beta\gamma\delta\mu\nu}_{abcd}\p^2_{\alpha\beta}Z^b u\p^2_{\gamma\delta}Z^c u\p^2_{\mu\nu}Z^d u\\
&+\sum_{\alpha,\beta,\gamma,\delta,\mu=0}^3\sum_{b+c+d\le a} G^{\alpha\beta\gamma\delta\mu}_{abcd}\p^2_{\alpha\beta}Z^b u\p^2_{\gamma\delta}Z^c u\p_{\mu}Z^d u\\
&+\sum_{\alpha,\beta,\gamma,\delta=0}^3 \sum_{b+c+d\le a}H^{\alpha\beta\gamma\delta}_{abcd}\p^2_{\alpha\beta}Z^b u\p_\gamma Z^cu\p_\delta Z^du
+\sum_{\alpha,\beta,\gamma=0}^3 \sum_{b+c+d\le a}S^{\alpha\beta\gamma}_{abcd}\p_{\alpha}Z^b u\p_\beta Z^c u\p_\gamma Z^du,
\end{split}
\end{equation}
where $F^{\alpha\beta\gamma\delta\mu\nu}_{abcd}$, $G^{\alpha\beta\gamma\delta\mu}_{abcd}$, $H^{\alpha\beta\gamma\delta}_{abcd}$ and $S^{\alpha\beta\gamma}_{abcd}$
are constants, moreover, $F^{\alpha\beta\gamma\delta\mu\nu}_{aa00}=F^{\alpha\beta\gamma\delta\mu\nu}_{a0a0}=F^{\alpha\beta\gamma\delta\mu\nu}_{a00a}=F^{\alpha\beta\gamma\delta\mu\nu}$,  $G^{\alpha\beta\gamma\delta\mu}_{aa00}=G^{\alpha\beta\gamma\delta\mu}_{a0a0}=G^{\alpha\beta\gamma\delta\mu}$, $H^{\alpha\beta\gamma\delta}_{aa00}=H^{\alpha\beta\gamma\delta}$.
In addition, if the partial null condition \eqref{null:def-cubic} holds, then for any $\xi=(\xi_0,\xi_1,\xi_2)\in\{\pm1,\SS^1\}$,
\begin{equation}\label{null:high-cubic}
\begin{split}
&\sum_{\alpha,\beta,\gamma,\delta,\mu,\nu=0}^2 F^{\alpha\beta\gamma\delta\mu\nu}_{abcd}\xi_\alpha\xi_\beta\xi_\gamma\xi_\delta\xi_\mu\xi_\nu\equiv0,
\quad\sum_{\alpha,\beta,\gamma,\delta,\mu=0}^2 G^{\alpha\beta\gamma\delta\mu}_{abcd}\xi_\alpha\xi_\beta\xi_\gamma\xi_\delta\xi_\mu\equiv0,\\
&\sum_{\alpha,\beta,\gamma,\delta=0}^2 H^{\alpha\beta\gamma\delta}_{abcd}\xi_\alpha\xi_\beta\xi_\gamma\xi_\delta \equiv0,
\quad\sum_{\alpha,\beta,\gamma=0}^2 S^{\alpha\beta\gamma}_{abcd}\xi_\alpha\xi_\beta\xi_\gamma \equiv0.
\end{split}
\end{equation}
\end{lemma}

The following estimates on the null condition structures will play important roles in establishing
the global classical solutions of problem \eqref{QWE} and problem \eqref{QWE-cubic}, respectively.
\begin{lemma}\label{lem:null}
Suppose that the constants $N_1^{\alpha\beta}$, $N_2^{\alpha\beta\gamma}$ and $N_3^{\alpha\beta\gamma\delta}$ satisfy
that for any $\xi=(\xi_0,\xi_1,\xi_2,\xi_3)\in\{\pm1,\SS^2\}$,
\begin{equation*}
\sum_{\alpha,\beta=0}^3N_1^{\alpha\beta}\xi_\alpha\xi_\beta\equiv0,
\quad\sum_{\alpha,\beta,\gamma=0}^3N_2^{\alpha\beta\gamma}\xi_\alpha\xi_\beta\xi_\gamma\equiv0,
\quad\sum_{\alpha,\beta,\gamma,\delta=0}^3N_3^{\alpha\beta\gamma\delta}
\xi_\alpha\xi_\beta\xi_\gamma\xi_\delta\equiv0.
\end{equation*}
Then for smooth functions $v,w$ and $|x|>0$, it holds that
\begin{equation}\label{null:structure}
\begin{split}
\Big|\sum_{\alpha,\beta=0}^3N_1^{\alpha\beta}\p_\alpha v\p_\beta w\Big|
&\ls|\bar\p v||\p w|+|\p v||\bar\p w|,\\
\Big|\sum_{\alpha,\beta,\gamma=0}^3N_2^{\alpha\beta\gamma}\p^2_{\alpha\beta}v\p_\gamma w\Big|
&\ls|\bar\p\p v||\p w|+|\p^2 v||\bar\p w|,\\
\Big|\sum_{\alpha,\beta,\gamma,\delta=0}^3N_3^{\alpha\beta\gamma\delta}
\p^2_{\alpha\beta}v\p^2_{\gamma\delta}w\Big|&\ls|\bar\p\p v||\p^2w|+|\p^2v||\bar\p\p w|,
\end{split}
\end{equation}
where $\ds|\bar\p v|:=\Big(\sum_{i=1}^3|\bar\p_iv|^2\Big)^{1/2}$.
\end{lemma}
\begin{lemma}\label{lem:null-cubic}
Suppose that the constants $N_1^{\alpha\beta\gamma\delta\mu\nu}$, $N_2^{\alpha\beta\gamma\delta\mu}$, $N_3^{\alpha\beta\gamma\delta}$
and $N_4^{\alpha\beta\gamma}$ satisfy that for any $\xi=(\xi_0,\xi_1,\xi_2,\xi_3)\in\{\pm1,\SS^2\}$,
\begin{equation*}
\begin{split}
&\sum_{\alpha,\beta,\gamma,\delta,\mu,\nu=0}^2 N_1^{\alpha\beta\gamma\delta\mu\nu}\xi_\alpha\xi_\beta\xi_\gamma\xi_\delta\xi_\mu\xi_\nu\equiv0,
\quad\sum_{\alpha,\beta,\gamma,\delta,\mu=0}^2N_2^{\alpha\beta\gamma\delta\mu}\xi_\alpha\xi_\beta\xi_\gamma\xi_\delta\xi_\mu\equiv0,\\
&\sum_{\alpha,\beta,\gamma,\delta=0}^2 N_3^{\alpha\beta\gamma\delta}\xi_\alpha\xi_\beta\xi_\gamma\xi_\delta\equiv0,
\quad\sum_{\alpha,\beta,\gamma=0}^2 N_4^{\alpha\beta\gamma}\xi_\alpha\xi_\beta\xi_\gamma\equiv0.
\end{split}
\end{equation*}
Then for smooth functions $h, v, w$ and $|x|>0$, it holds that
\begin{equation}\label{null:structure-cubic}
\begin{split}
\Big|\sum_{\alpha,\beta,\gamma,\delta,\mu,\nu=0}^2N_1^{\alpha\beta\gamma\delta\mu\nu}
\p^2_{\alpha\beta}h\p^2_{\gamma\delta}v\p^2_{\mu\nu}w\Big|
&\ls|\bar\p\p h||\p^2v||\p^2w|+|\p^2h||\bar\p\p v||\p^2w|+|\p^2h||\p^2v||\bar\p\p w|,\\
\Big|\sum_{\alpha,\beta,\gamma,\delta,\mu=0}^2N_2^{\alpha\beta\gamma\delta\mu}
\p^2_{\alpha\beta}h \p^2_{\gamma\delta}v\p_{\mu}w\Big|
&\ls|\bar\p\p h||\p^2v||\p w|+|\p^2h||\bar\p\p v||\p w|+|\p^2h||\p^2v||\bar\p w|,\\
\Big|\sum_{\alpha,\beta,\gamma,\delta=0}^2N_3^{\alpha\beta\gamma\delta}
\p^2_{\alpha\beta}h \p_{\gamma}v\p_{\delta}w\Big|
&\ls|\bar\p\p h||\p v||\p w|+|\p^2h||\bar\p v||\p w|+|\p^2h||\p v||\bar\p w|,\\
\Big|\sum_{\alpha,\beta,\gamma=0}^2N_4^{\alpha\beta\gamma}
\p_{\alpha}h\p_{\beta}v\p_{\gamma}w\Big|
&\ls|\bar\p h||\p v||\p w|+|\p h||\bar\p v||\p w|+|\p h||\p v||\bar\p w|
\end{split}
\end{equation}
where $\ds|\bar\p h|:=\Big(\sum_{i=1}^2|\bar\p_ih|^2\Big)^{1/2}$.
\end{lemma}
\begin{proof}
Proofs of Lemma \ref{lem:null} and Lemma \ref{lem:null-cubic} can be refereed to the one in Lemma 2.2 of \cite{HouYin20jde}.
\end{proof}

\subsection{Bootstrap assumptions}\label{sect2-2}

Define the energy
\begin{equation*}
E_k(t):=\sum_{|a|\le k}\|\p Z^au\|_{L^2_{x,y}}.
\end{equation*}
In the rest of the paper, we will make the following bootstrap assumptions:

\underline{For problem \eqref{QWE}, we assume that }
\begin{equation}\label{bootstrap}
\begin{split}
E_N(t)+\cX_N(t)&\le\ve_1(1+t)^{\ve_2},\\
E_{N-10}(t)+\cX_{N-10}(t)&\le\ve_1\le1,\\
\w{t+|x|}^{1.4}|(Z^au)_n(t,x)|&\le|n|^{|a|+7-N}\ve_1,\quad|a|\le N-8,n\in\Z_*;
\end{split}
\end{equation}

\underline{For problem \eqref{QWE-cubic}, we assume that}
\begin{equation}\label{bootstrap-cubic}
\begin{split}
E_N(t)+\cX_N(t)&\le\ve_1(1+t)^{\ve_2},\\
E_{N-9}(t)+\cX_{N-9}(t)&\le\ve_1\leq 1,\\
\w{t+|x|}^{0.9}|(Z^au)_n(t,x)|&\le|n|^{|a|+5-N}\ve_1,\quad|a|\le N-7,n\in\Z_*,
\end{split}
\end{equation}
where $\ve_1>0 $ and $\ve_2\in [0,1/10]$ will be determined later.

\vskip 0.3 true cm

According to the definition of $E_N(t)$ and \eqref{bootstrap}-\eqref{bootstrap-cubic}, one has
\begin{equation}\label{bootstrap'}
\|(\p Z^au)_n(t,x)\|_{L^2_x}\ls|n|^{|a|-N}\ve_1\w{t}^{\ve_2},\quad|a|\le N,n\in\Z_*.
\end{equation}

Based on the bootstrap estimates \eqref{bootstrap}-\eqref{bootstrap-cubic}, we will derive some stronger estimates
of the solutions $u$ to problems \eqref{QWE} or \eqref{QWE-cubic} so that
\eqref{bootstrap}-\eqref{bootstrap-cubic} can be closed.

\section{The pointwise estimates of the zero mode}\label{sect3}

\subsection{The pointwise estimates of the zero mode on $\R^3\times\T$}
In this subsection, we first give two basic lemmas, from which a series of the pointwise estimates on the zero mode
are derived.
\begin{lemma}
For any sufficiently smooth function $f(t,x)$, it holds that
\begin{equation}\label{Linfty}
\|f(t,x)\|_{L^\infty(\R^3)}\ls(1+t)^{-3/2}\|f\|_{L^2(\R^3)}+\sum_{i=1,2}\|\nabla^i_xf\|_{L^2(\R^3)}.
\end{equation}
\end{lemma}
\begin{proof}
For any fixed $t\ge0$, denote $J=[\log_2(2+t)]+1>0$.
Let $\dot\Delta_j$ be the homogeneous dyadic blocks of the frequency on $\R^3$. Then it follows from the Bernstein inequality
and direct computation that
\begin{equation*}
\begin{split}
\|f\|_{L^\infty(\R^3)}&\ls\sum_{j\in\Z}\|\dot\Delta_jf\|_{L^\infty(\R^3)}
\ls\sum_{j\in\Z}2^{3j/2}\|\dot\Delta_jf\|_{L^2(\R^3)}\\
&\ls\sum_{j\le-J}2^{3j/2}\|\dot\Delta_jf\|_{L^2(\R^3)}+\sum_{-J\le j\le0}2^{3j/2}\|\dot\Delta_jf\|_{L^2(\R^3)}
+\sum_{j\ge0}2^{3j/2}\|\dot\Delta_jf\|_{L^2(\R^3)}\\
&\ls2^{-3J/2}\|f\|_{L^2(\R^3)}+\sum_{-J\le j\le0}2^{j/2}\|\nabla_xf\|_{L^2(\R^3)}
+\sum_{j\ge0}2^{-j/2}\|\nabla_x^2f\|_{L^2(\R^3)}\\
&\ls(1+t)^{-3/2}\|f\|_{L^2(\R^3)}+\|\nabla_xf\|_{L^2(\R^3)}+\|\nabla_x^2f\|_{L^2(\R^3)},
\end{split}
\end{equation*}
which yields \eqref{Linfty}.
\end{proof}

\begin{lemma}\label{Weighted-Sobelev}
For any sufficiently smooth function $f(t,x)$, we have that
\begin{equation}\label{Sobolev:ineq}
\w{x}\w{t\pm|x|}^{1/2}|f(t,x)|\ls\|\Omega^{\le2}f\|_{L^2(\R^3)}
+\sum_{|a|+|b|\le1}\|\w{t\pm|x|}\nabla_x\nabla_x^a\Omega^bf\|_{L^2(\R^3)}.
\end{equation}
\end{lemma}
\begin{proof}
Recalling (3.19) in \cite{Sideris00}, one has that for smooth function $R(r)$,
\begin{equation}\label{Sobolev:ineq1}
\begin{split}
&\Big(|x|^4|R(r)|^2\int_{\SS^2}|f(t,|x|\omega)|^4d\omega\Big)^{1/4}\\
\ls&\Big(\int_{|x'|\ge|x|}[|R(r)|^2|\p_rf(t,x')|^2+|R'(r)|^2|f(t,x')|^2]dx'\Big)^{1/4}
\Big(\int_{|x'|\ge|x|}|\Omega^{\le1}f(t,x')|^2dx'\Big)^{1/4},
\end{split}
\end{equation}
where $r=|x|>0$.
Choosing $R(r)=\w{t\pm r}$ in \eqref{Sobolev:ineq1} and applying the Sobolev embedding $W^{1,4}(\SS^2)\hookrightarrow L^\infty(\SS^2)$
to obtain
\begin{equation}\label{Sobolev:ineq2}
\begin{split}
&\quad\;|x|\w{t\pm|x|}^{1/2}|f(t,x)|\ls
\Big(|x|^4|\w{t\pm|x|}|^2\int_{\SS^2}|\Omega^{\le1}f(t,|x|\omega)|^4d\omega\Big)^{1/4}\\
&\ls\|\Omega^{\le2}f(t,x')\|_{L^2(\R^3:|x'|\ge|x|)}
+\|\w{t\pm|x'|}\p_r\Omega^{\le1}f(t,x')\|_{L^2(\R^3:|x'|\ge|x|)}.
\end{split}
\end{equation}

When $|x|\ge1/3$, \eqref{Sobolev:ineq2} implies \eqref{Sobolev:ineq} directly.

Next, we prove \eqref{Sobolev:ineq} for $|x|\le1/3$.
Define the cutoff function $\chi(s)\in C^\infty(\R)$ which satisfies
\begin{equation}\label{chi:def}
0\le\chi\le1,\qquad\chi(s)=\left\{
\begin{aligned}
&1,\quad s\le1/3,\\
&0,\quad s\ge1/2.
\end{aligned}
\right.
\end{equation}
Applying the Sobolev embedding $H^2(\R^3)\hookrightarrow L^\infty(\R^3)$ to $\ds\chi(|x|)f(t,x)$ yields
\begin{equation}\label{Sobolev:ineq3}
|\chi(|x|)f(t,x)|\ls\|f(t,x)\|_{L^2(\R^3:|x|\le1)}
+\w{t}^{-1}\sum_{i=1,2}\|\w{t\pm|x|}\nabla_x^if(t,x)\|_{L^2(\R^3:|x|\le1)}.
\end{equation}
Utilizing \eqref{Sobolev:ineq2} to the first term on the right hand side of \eqref{Sobolev:ineq3}, we achieve
\begin{equation*}
\begin{split}
\w{t}^{1/2}\|f(t,x)\|_{L^2(\R^3:|x|\le1)}&\ls\Big\||x|^{-1}\Big\|_{L^2(\R^3:|x|\le1)}
\||x|\w{t\pm|x|}^{1/2} f(t,x)\|_{L^\infty(\R^3:0<|x|\le1)}\\
&\ls\|\Omega^{\le2}f\|_{L^2(\R^3)}
+\sum_{|a|+|b|\le1}\|\w{t\pm|x|}\nabla_x\nabla_x^a\Omega^bf\|_{L^2(\R^3)}.
\end{split}
\end{equation*}
This, together with \eqref{Sobolev:ineq3}, leads to \eqref{Sobolev:ineq} for $|x|\le1/3$.
\end{proof}

\begin{lemma}
For any multi-index $a$ with $|a|\le N-2$, it holds that
\begin{equation}\label{pw:zero}
\w{x}\w{t-|x|}^{1/2}|P_{=0}\p Z^au(t,x,y)|\ls E_{|a|+2}(t)+\cX_{|a|+2}(t).
\end{equation}
In addition, one has that in the region $|x|\ge\w{t}/2$,
\begin{equation}\label{pw:good}
\w{x}\w{t+|x|}^{1/2}|P_{=0}\bar\p Z^au(t,x,y)|\ls E_{|a|+2}(t)+\cX_{|a|+2}(t);\
\end{equation}
and in the region $|x|\le\w{t}/3$,
\begin{equation}\label{pw:awaycone}
\w{t}|P_{=0}\p Z^au(t,x,y)|\ls E_{|a|+2}(t)+\cX_{|a|+2}(t).
\end{equation}
\end{lemma}
\begin{proof}
For any smooth function $g(t,x,y)$, we have
\begin{equation*}
|P_{=0}g(t,x,y)|\ls\|P_{=0}g(t,x,y)\|_{L^2(\T)}.
\end{equation*}
By the definition of $\cX_{|a|+2}(t)$ with $d=3$, choosing $f=P_{=0}\p Z^au$ and $f=(1-\chi(\frac{|x|}{\w{t}}))P_{=0}\bar\p Z^au$ in \eqref{Sobolev:ineq}, respectively, and subsequently taking the $L^2_y$ norms on the both sides of the resulted inequalities,
then  \eqref{pw:zero} and \eqref{pw:good} are obtained, respectively.
In addition, applying \eqref{Linfty} to $\chi(\frac{|x|}{\w{t}})P_{=0}\p Z^au(t,x,y)$ yields \eqref{pw:awaycone}.
\end{proof}

\begin{corollary}
Under the bootstrap assumptions \eqref{bootstrap}, for any multi-index $a$ with $|a|\le N-12$, it holds that
\begin{equation}\label{pw:zero'}
\w{t+|x|}|P_{=0}\p Z^au(t,x,y)|\ls\ve_1.
\end{equation}
\end{corollary}
\begin{proof}
\eqref{pw:zero'} follows from \eqref{bootstrap}, \eqref{pw:zero} and \eqref{pw:awaycone}.
\end{proof}

\subsection{The pointwise estimates of the zero mode on $\R^2\times\T$}

Although the structure of this subsection is completely similar to that of Subsection 3.1, due to the different space dimensions
and for readers' convenience, the related details are still given.

\begin{lemma}
For any sufficiently smooth function $f(t,x)$, it holds that
\begin{equation}\label{Linfty-cubic}
\|f(t,x)\|_{L^\infty(\R^2)}\ls~\mathrm{ln}^{1/2}(e+t)\|\nabla_x f\|_{L^2(\R^2)}+\w{t}^{-1}(\|f\|_{L^2(\R^2)}+\|\nabla^2_x f\|_{L^2(\R^2)}).
\end{equation}
\end{lemma}
\begin{proof}
The proof can be found in \cite[(3.4)]{Lei16}.
\end{proof}

\begin{lemma}
For any sufficiently smooth function $f(t,x)$, one has
\begin{equation}\label{Sobolev:ineq-cubic}
\w{x}^{1/2} \w{t\pm|x|}^{1/2}|f(t,x)|\ls\|\Omega^{\le 1}f\|_{L^2(\R^2)}+\sum_{|a|+|b|\le1}\|\w{t\pm|x|}\nabla_x\nabla_x^a\Omega^bf\|_{L^2(\R^2)}.
\end{equation}
\end{lemma}

\begin{proof}
By Sobolev's embedding $H^1(\SS^1)\hookrightarrow L^\infty(\SS^1)$ on the circle, we calculate
\begin{equation}\label{Sobolev:ineq1-cubic}
\begin{split}
&|x|\w{t\pm |x|}|f(t,|x|\frac{x}{|x|})|^2\ls |x|\w{t\pm |x|}\int_{\SS^1}|\Omega^{\le1}f(t,|x|\omega)|^2d\omega\\
=&-\int_{\SS^1}d\omega \int^{\infty}_{|x|}|x|\p_\rho[\w{t\pm \rho}|\Omega^{\le1}f(t,\rho\omega)|^2]d\rho\\
\ls& \int_{\SS^1}d\omega \int^{\infty}_{|x|}[\w{t\pm \rho}|\Omega^{\le1}f(t,\rho\omega)||\partial_\rho\Omega^{\le1}f(t,\rho\omega)|+
|\Omega^{\le1}f(t,\rho\omega)|^2]\rho d\rho\\
\ls& \int_{\SS^1}d\omega \int^{\infty}_{|x|}[\w{t\pm \rho}^2|\partial_\rho\Omega^{\le1}f(t,\rho\omega)|^2+
|\Omega^{\le1}f(t,\rho\omega)|^2]\rho d\rho\\
\ls&~\|\Omega^{\le1}f(t,x')\|^2_{L^2(\R^2:|x'|\ge|x|)}
+\|\w{t\pm|x'|}\nabla_x\Omega^{\le1}f(t,x')\|^2_{L^2(\R^2:|x'|\ge|x|)}.
\end{split}
\end{equation}

For $|x|\ge1/3$, \eqref{Sobolev:ineq1-cubic} yields \eqref{Sobolev:ineq-cubic}.

Next, we prove \eqref{Sobolev:ineq-cubic} for $|x|\le1/3$.
As in the proof of Lemma \ref{Weighted-Sobelev}, applying the Sobolev embedding $H^2(\R^2)\hookrightarrow L^\infty(\R^2)$ to $\ds\chi(|x|)f(t,x)$
derives
\begin{equation}\label{Sobolev:ineq3-cubic}
|\chi(|x|)f(t,x)|\ls\|f(t,x)\|_{L^2(\R^2:|x|\le1)}
+\w{t}^{-1}\sum_{i=1,2}\|\w{t\pm|x|}\nabla_x^if(t,x)\|_{L^2(\R^2:|x|\le1)},
\end{equation}
where $\chi$ is defined by \eqref{chi:def}.
By using \eqref{Sobolev:ineq1-cubic} to the first term on the right hand side of \eqref{Sobolev:ineq3-cubic}, we arrive at
\begin{equation*}
\begin{split}
\w{t}^{1/2}\|f(t,x)\|_{L^2(\R^2:|x|\le1)}&\ls \||x|^{-1/2}\|_{L^2(\R^2:|x|\le1)}
\||x|^{1/2}\w{t\pm|x|}^{1/2}f(t,x)\|_{L^\infty(\R^2:|x|\le1)}\\
&\ls\|\Omega^{\le 1}f\|_{L^2(\R^2)}+\|\w{t\pm|x|}\nabla_x\Omega^{\leq 1}f\|_{L^2(\R^2)}.
\end{split}
\end{equation*}
This, together with \eqref{Sobolev:ineq3-cubic}, leads to \eqref{Sobolev:ineq-cubic} for $|x|\le1/3$.
\end{proof}

\begin{lemma}
For any multi-index $a$ with $|a|\le N-2$, it holds that
\begin{equation}\label{pw:zero-cubic}
\w{x}^{1/2} \w{t-|x|}^{1/2} |P_{=0}\p Z^au(t,x,y)|\ls E_{|a|+2}(t)+\cX_{|a|+2}(t).
\end{equation}
Furthermore, one has that in the region $|x|\ge\w{t}/2$,
\begin{equation}\label{pw:good-cubic}
\w{x}^{1/2} \w{t+|x|}^{1/2} |P_{=0}\bar\p Z^au(t,x,y)|\ls E_{|a|+2}(t)+\cX_{|a|+2}(t);\
\end{equation}
and in the region $|x|\le\w{t}/3$,
\begin{equation}\label{pw:awaycone-cubic}
\w{t}\mathrm{ln}^{-1/2}(e+t) |P_{=0}\p Z^au(t,x,y)|\ls E_{|a|+2}(t)+\cX_{|a|+2}(t).
\end{equation}
\end{lemma}
\begin{proof}
By the definition of $\cX_{|a|+2}(t)$ with $d=2$, choosing $f=P_{=0}\p Z^au$ and $f=(1-\chi(\frac{|x|}{\w{t}}))P_{=0}\bar\p Z^au$ in \eqref{Sobolev:ineq-cubic}, respectively, and taking the $L^2_y$ norms on the both sides of the resulted inequalities, then
we can get \eqref{pw:zero-cubic} and \eqref{pw:good-cubic}, respectively.
Applying \eqref{Linfty-cubic} to $\chi(\frac{|x|}{\w{t}})P_{=0}\p Z^au(t,x,y)$ derives \eqref{pw:awaycone-cubic}.
\end{proof}

\begin{corollary}
Under the bootstrap assumptions \eqref{bootstrap-cubic}, for any multi-index $a$ with $|a|\le N-11$, it holds that
\begin{equation}\label{pw:zero'-cubic}
\w{t+|x|}^{1/2} |P_{=0}\p Z^au(t,x,y)|\ls\ve_1.
\end{equation}
\end{corollary}
\begin{proof}
\eqref{pw:zero'-cubic} follows from \eqref{bootstrap-cubic}, \eqref{pw:zero-cubic} and \eqref{pw:awaycone-cubic}.
\end{proof}

\section{The pointwise estimates of the non-zero modes}\label{sect4}
\subsection{The pointwise estimates of the non-zero modes on $\R^3\times\T$}
\begin{lemma}\label{lem:pwL2KG}
Let $v$ be a solution to $\Box_{t,x}v+v=\cF(t,x)$. Then one has
\begin{equation}\label{pwL2:KG}
\begin{split}
\w{t+|x|}^{1.4}|v(t,x)|&\ls\sum_{|a|\le5}\sup_{\tau\in[0,t]}\w{\tau}^{-1/10}
\|\w{\tau+|\cdot|}\Gamma^a\cF(\tau,\cdot)\|_{L^2(\R^3)}\\
&\quad+\sum_{|a|\le6}\|\w{\cdot}^2\Gamma^av(0,\cdot)\|_{L^2(\R^3)}
\end{split}
\end{equation}
and
\begin{equation}\label{pwL2:KG'}
\begin{split}
\w{t+|x|}^{3/2}|v(t,x)|&\ls\sum_{|a|\le5}\sup_{\tau\in[0,t]}\w{\tau}^{1/10}
\|\w{\tau+|\cdot|}\Gamma^a\cF(\tau,\cdot)\|_{L^2(\R^3)}\\
&\quad+\sum_{|a|\le6}\|\w{\cdot}^2\Gamma^av(0,\cdot)\|_{L^2(\R^3)}.
\end{split}
\end{equation}
\end{lemma}
\begin{proof}
It follows  from Theorem 1 in \cite{Georgiev92} that
\begin{equation}\label{pwL2:KG1}
\begin{split}
\w{t+|x|}^{3/2}|v(t,x)|&\ls\sum_{j=0}^\infty\sum_{|a|\le5}\sup_{\tau\in[0,t]}
\vp_j(\tau)\|\w{\tau+|\cdot|}\Gamma^a\cF(\tau,\cdot)\|_{L^2(\R^3)}\\
&\quad+\sum_{j=0}^\infty\sum_{|a|\le6}\|\w{x'}^{3/2}\vp_j(|x'|)\Gamma^av(0,x')\|_{L^2(\R^3)},
\end{split}
\end{equation}
where $\{\vp_j\}_{j=0}^\infty$ is a Littlewood-Paley partition of unit satisfying
\begin{equation}\label{pwL2-phi-def}
\begin{split}
\sum_{j=0}^\infty\vp_j(\tau)=1, \tau\ge0,\quad \vp_j\in C_0^\infty(\R), \vp_j\ge0, j\ge0,\\
\supp\vp_j=[2^{j-1},2^{j+1}], j\ge1,\quad \supp\vp_0\cap\R_+=[0,2].
\end{split}
\end{equation}
Let $j_0:=[\log_2\w{t}]+2$. Then $\vp_j(\tau)\equiv0$ holds for any $\tau\in[0,t]$ and $j>j_0$.
Thereby, we arrive at
\begin{equation}\label{pwL2:KG2}
\begin{split}
&\sum_{j=0}^\infty\sum_{|a|\le5}\sup_{\tau\in[0,t]}\vp_j(\tau)\|\w{\tau+|\cdot|}\Gamma^a\cF\|_{L^2(\R^3)}\\
&\ls\Big(\sum_{j=0}^{j_0}\sup_{\tau\in[0,t]}\vp_j(\tau)\w{\tau}^{1/10}\Big)
\sum_{|a|\le5}\sup_{\tau\in[0,t]}\w{\tau}^{-1/10}\|\w{\tau+|\cdot|}\Gamma^a\cF\|_{L^2(\R^3)}\\
&\ls\Big(\sum_{j=0}^{j_0}2^{j/10}\Big)\sum_{|a|\le5}\sup_{\tau\in[0,t]}
\w{\tau}^{-1/10}\|\w{\tau+|\cdot|}\Gamma^a\cF\|_{L^2(\R^3)}\\
&\ls\w{t}^{1/10}\sum_{|a|\le5}\sup_{\tau\in[0,t]}
\w{\tau}^{-1/10}\|\w{\tau+|\cdot|}\Gamma^a\cF\|_{L^2(\R^3)}.
\end{split}
\end{equation}
On the other hand, it is easy to get that
\begin{equation}\label{pwL2:KG3}
\begin{split}
\sum_{j=0}^\infty\sum_{|a|\le6}\|\w{x'}^{3/2}\vp_j(|x'|)\Gamma^av(0,x')\|_{L^2(\R^3)}
\ls\sum_{|a|\le6}\|\w{\cdot}^2\Gamma^av(0,\cdot)\|_{L^2(\R^3)}.
\end{split}
\end{equation}
Substituting \eqref{pwL2:KG2} and \eqref{pwL2:KG3} into \eqref{pwL2:KG1} yields \eqref{pwL2:KG}.
The proof of \eqref{pwL2:KG'} is analogous.
\end{proof}

\begin{lemma}\label{lem:pw:nonzero}
Let $u$ be the solution to \eqref{QWE}.
Under the bootstrap assumptions \eqref{bootstrap}, for any $n\in\Z_*$ and multi-index $a$ with $|a|\le N-8$, it holds that
\begin{equation}\label{pw:nonzero}
\w{t+|x|}^{1.4}|(Z^au)_n(t,x)|\ls|n|^{|a|+13/2-N}(\ve+\ve_1^2)
\end{equation}
and
\begin{equation}\label{dyu:pw'}
\w{t+|x|}^{3/2}|(\p_yu)_n(t,x)|\ls|n|^{-9}(\ve+\ve_1^2).
\end{equation}
\end{lemma}
\begin{proof}
Taking the Fourier transformation on the both sides of \eqref{eqn:high} for the variable $y$ to obtain
\begin{equation}\label{pw:nonzero1}
(\Box_{t,x}+|n|^2)(Z^au)_n(t,x)=(Q^a)_n(t,x).
\end{equation}
Setting $(s,z)=(|n|t,|n|x)$, $\Box_{s,z}=\p_s^2-\Delta_z$ and $(\overline{Z^au})_n(s,z):=(Z^au)_n(\frac{s}{|n|},\frac{z}{|n|})$,
then \eqref{pw:nonzero1} is changed to
\begin{equation*}
|n|^2(\Box_{s,z}+1)(\overline{Z^au})_n(s,z)=(\overline{Q^a})_n(s,z):=(Q^a)_n(\frac{s}{|n|},\frac{z}{|n|}).
\end{equation*}
Thus, it can be derived from \eqref{pwL2:KG} that
\begin{equation*}
\begin{split}
\w{s+|z|}^{1.4}|(\overline{Z^au})_n(s,z)|
&\ls|n|^{-2}\sum_{|a'|\le5}\sum_{\tilde\Gamma\in\{\p_{\tau,z'},L,\Omega\}}
\sup_{\tau\in[0,s]}\w{\tau}^{-0.1}\|\w{\tau+|z'|}\tilde\Gamma^{a'}(\overline{Q^a})_n(\tau,z')\|_{L_{z'}^2(\R^3)}\\
&\quad+|n|^{-2}\sum_{|a'|\le6}\sum_{\tilde\Gamma\in\{\p_{\tau,z'},L,\Omega\}}
\|\w{z'}^2\tilde\Gamma^{a'}(\overline{Z^au})_n(0,z')\|_{L_{z'}^2(\R^3)},
\end{split}
\end{equation*}
where $\tilde\Gamma'$s are the vector fields in $(\tau,z')$ variables.
Note that $\w{t+|x|}\ls\w{n(t+|x|)}=\w{s+|z|}$, and the vectors $L, \Omega$ are homogeneous
in $(s,z)$ and are scaling invariant.
Then returning to the $(t,x)$ variables, we find that
\begin{equation}\label{pw:nonzero2}
\begin{split}
\w{t+|x|}^{1.4}|(Z^au)_n(t,x)|&\ls|n|^{1/2}\sum_{|a'|\le5}
\sup_{\tau\in[0,t]}\w{\tau}^{-0.1}\|\w{\tau+|x|}\Gamma^{a'}(Q^a)_n(\tau, x)\|_{L^2_x(\R^3)}\\
&\quad+|n|^{3/2}\sum_{|a'|\le6}\|\w{x}^2\Gamma^{a'}(Z^au)_n(0,x)\|_{L^2_x(\R^3)}.
\end{split}
\end{equation}
In addition, the definition of $Q^a$ in \eqref{eqn:high} leads to
\begin{equation}\label{pw:nonzero3}
\sum_{|a'|\le5}\|\w{\tau+|x|}\Gamma^{a'}(Q^a)_n(\tau, x)\|_{L_x^2}
\ls\sum_{|b|+|c|\le|a|+5}\Big(J_n^{bc}+\sum_{m\in\Z_*,m\neq n}J_{mn}^{bc}\Big),
\end{equation}
where
\begin{equation*}
\begin{split}
J_n^{bc}&:=\|\w{\tau+|x|}(\p\p^{\le1}Z^bu)_n(P_{=0}\p\p^{\le1}Z^cu)\|_{L_x^2},\\
J_{mn}^{bc}&:=\|\w{\tau+|x|}(\p\p^{\le1}Z^bu)_{n-m}(\p\p^{\le1}Z^cu)_m\|_{L_x^2}.
\end{split}
\end{equation*}
For $J_n^{bc}$ with $|b|\le|a|-3$, the assumptions \eqref{bootstrap} imply
\begin{equation*}
\begin{split}
\sum_{\substack{|b|+|c|\le|a|+5,\\|b|\le|a|-3}}J_n^{bc}
&\ls\sum_{\substack{|b|+|c|\le|a|+5,\\|b|\le|a|-3}}
\|\w{\tau+|x|}(\p\p^{\le1}Z^bu)_n\|_{L^\infty}\|P_{=0}\p\p^{\le1}Z^cu\|_{L_x^2}\\
&\ls\sum_{|b|\le|a|-3}|n|^{|b|+9-N}\ve_1^2\w{\tau}^{\ve_2-0.4}
\ls|n|^{|a|+6-N}\ve_1^2\w{\tau}^{-0.3}.
\end{split}
\end{equation*}
For $J_n^{bc}$ with $|b|\ge|a|-2$, we can see that $|c|\le7\le N-13$ holds.
Then \eqref{bootstrap'} and \eqref{pw:zero'} ensure that
\begin{equation*}
\begin{split}
\sum_{\substack{|b|+|c|\le|a|+5,\\|a|-2\le|b|\le|a|+5}}J_n^{bc}
&\ls\sum_{\substack{|b|+|c|\le|a|+5,\\|a|-2\le|b|\le|a|+5}}
\|(\p\p^{\le1}Z^bu)_n\|_{L_x^2}\|\w{\tau+|x|}P_{=0}\p\p^{\le1}Z^cu\|_{L^\infty}\\
&\ls\sum_{|b|\le|a|+5}|n|^{|b|+1-N}\ve_1^2\w{\tau}^{\ve_2}\ls|n|^{|a|+6-N}\ve_1^2\w{\tau}^{\ve_2}.
\end{split}
\end{equation*}
Thus one can obtain
\begin{equation}\label{pw:nonzero4}
\sum_{|b|+|c|\le|a|+5}J_n^{bc}\ls|n|^{|a|+6-N}\ve_1^2\w{\tau}^{\ve_2}.
\end{equation}
For $J_{mn}^{bc}$ with $|b|\le|a|-3$, we conclude from \eqref{bootstrap} and \eqref{bootstrap'} that
\begin{equation}\label{pw:nonzero5}
\begin{split}
\sum_{\substack{|b|+|c|\le|a|+5,\\|b|\le|a|-3}}J_{mn}^{bc}
&\ls\sum_{\substack{|b|+|c|\le|a|+5,\\|b|\le|a|-3}}
\|\w{\tau+|x|}(\p\p^{\le1}Z^bu)_{n-m}\|_{L^\infty}\|(\p\p^{\le1}Z^cu)_m\|_{L_x^2}\\
&\ls\ve_1^2\sum_{\substack{|b|+|c|\le|a|+5,\\|b|\le|a|-3}}
|n-m|^{|b|+9-N}|m|^{|c|+1-N}\w{\tau}^{\ve_2-0.4}\\
&\ls\ve_1^2|n-m|^{|a|+6-N}|m|^{|a|+6-N}\w{\tau}^{-0.3}.
\end{split}
\end{equation}
For $J_{mn}^{bc}$ with $|b|\ge|a|-2$ and $|a|\ge10$, then $|c|\le7\le N-10$ holds.
Thus, we achieve
\begin{equation}\label{pw:nonzero6}
\begin{split}
\sum_{\substack{|b|+|c|\le|a|+5,\\|a|-2\le|b|\le|a|+5}}J_{mn}^{bc}
&\ls\sum_{\substack{|b|+|c|\le|a|+5,\\|a|-2\le|b|\le|a|+5}}
\|(\p\p^{\le1}Z^bu)_{n-m}\|_{L_x^2}\|\w{\tau+|x|}(\p\p^{\le1}Z^cu)_m\|_{L^\infty}\\
&\ls\ve_1^2\sum_{\substack{|b|+|c|\le|a|+5,\\|a|-2\le|b|\le|a|+5}}|n-m|^{|b|+1-N}|m|^{|c|+9-N}\w{\tau}^{-0.3}\\
&\ls\ve_1^2|n-m|^{|a|+6-N}|m|^{16-N}\w{\tau}^{-0.3}\\
&\ls\ve_1^2|n-m|^{|a|+6-N}|m|^{|a|+6-N}\w{\tau}^{-0.3}.
\end{split}
\end{equation}
Thereafter, for $10\le|a|\le N-8$, collecting \eqref{pw:nonzero5} and \eqref{pw:nonzero6} yields
\begin{equation}\label{pw:nonzero7}
\sum_{|b|+|c|\le|a|+5}\sum_{m\in\Z_*,m\neq n}J_{mn}^{bc}\ls\ve_1^2|n|^{|a|+6-N}\w{\tau}^{-0.3}.
\end{equation}
If $|a|\le9$, it is easy to find that $|b|,|c|\le|a|+5\le14\le N-10$.
Then we have
\begin{equation}\label{pw:nonzero8}
\begin{split}
J_{mn}^{bc}&\ls\ve_1^2\w{\tau}^{-0.3}\min\{|n-m|^{|b|+9-N}|m|^{|c|+1-N},|n-m|^{|b|+1-N}|m|^{|c|+9-N}\}\\
&\ls\ve_1^2\w{\tau}^{-0.3}\min\{|n-m|^{23-N}|m|^{|a|+6-N},|n-m|^{|a|+6-N}|m|^{23-N}\}.
\end{split}
\end{equation}
For $|a|\le9$, it is deduced from \eqref{pw:nonzero8} and $N\ge25$ that
\begin{equation}\label{pw:nonzero9}
\begin{split}
&\w{\tau}^{0.3}\sum_{m\in\Z_*,m\neq n}J_{mn}^{bc}\\
&\ls\ve_1^2\sum_{\substack{m\in\Z_*,m\neq n,\\|m|\ge|n|/2}}|n-m|^{23-N}|m|^{|a|+6-N}
+\ve_1^2\sum_{\substack{m\in\Z_*,m\neq n,\\|m|\le|n|/2}}|n-m|^{|a|+6-N}|m|^{23-N}\\
&\ls\ve_1^2|n|^{|a|+6-N}\sum_{m\in\Z_*}|m|^{23-N}
\ls\ve_1^2|n|^{|a|+6-N}.
\end{split}
\end{equation}
For the second line of \eqref{pw:nonzero2}, \eqref{initial:data} leads to
\begin{equation}\label{pw:nonzero10}
\sum_{|a'|\le6}\|\w{x}^2\Gamma^{a'}(Z^au)_n(0,x)\|_{L^2_x(\R^3)}\ls\ve|n|^{|a|+6-N-1}.
\end{equation}
Inserting \eqref{pw:nonzero3}, \eqref{pw:nonzero4}, \eqref{pw:nonzero7}, \eqref{pw:nonzero9} and \eqref{pw:nonzero10}
into \eqref{pw:nonzero2} derives \eqref{pw:nonzero}.

At last, we turn to the proof of \eqref{dyu:pw'}.
Let $Z^a=\p_y$ and repeat the steps to obtain \eqref{pw:nonzero2} with \eqref{pwL2:KG'} instead of \eqref{pwL2:KG}.
Then we arrive at
\begin{equation}\label{pw:nonzero11}
\w{t+|x|}^{3/2}|(Z^au)_n(t,x)|\ls|n|^{3/5}\sum_{|a'|\le6}
\sup_{\tau\in[0,t]}\w{\tau}^{0.1}\|\w{\tau+|x|}(Q^{a'})_n(\tau, x)\|_{L^2_x(\R^3)}+\ve|n|^{15/2-N},
\end{equation}
where \eqref{pw:nonzero10} for $|a|=1$ has been used.
In addition, by \eqref{bootstrap} and \eqref{pw:nonzero9}, one then has
\begin{equation*}
\begin{split}
&\sum_{|a'|\le6}\|\w{\tau+|x|}(Q^{a'})_n(\tau, x)\|_{L^2_x(\R^3)}\\
&\ls\ve_1^2|n|^{7-N}\w{\tau}^{-0.3}
+\sum_{|b|,|c|\le7}\|\w{\tau+|x|}(\p Z^bu)_n\|_{L^\infty}\|P_{=0}\p Z^cu\|_{L_x^2}\\
&\ls\ve_1^2|n|^{15-N}\w{\tau}^{-0.3}.
\end{split}
\end{equation*}
This, together with \eqref{pw:nonzero11} and $N\ge25$, yields \eqref{dyu:pw'}.
\end{proof}

\begin{corollary}
Let $u$ be the solution of \eqref{QWE}.
Under the bootstrap assumptions \eqref{bootstrap}, for any multi-index $a$ with $|a|\le N-8$, it holds that
\begin{equation}\label{pw:nonzero'}
|P_{\neq0}Z^au(t,x,y)|\ls\w{t+|x|}^{-1.4}(\ve+\ve_1^2).
\end{equation}
\end{corollary}
\begin{proof}
For $|a|\le N-8$, it follows from \eqref{pw:nonzero} that
\begin{equation*}
|P_{\neq0}Z^au(t,x,y)|=\Big|\sum_{n\in\Z_*}e^{{\rm i}ny}(Z^au)_n(t,x)\Big|
\ls\w{t+|x|}^{-1.4}(\ve+\ve_1^2)\sum_{n\in\Z_*}|n|^{-3/2},
\end{equation*}
which yields \eqref{pw:nonzero'}.
\end{proof}

\subsection{The pointwise estimates of the non-zero modes on $\R^2\times\T$}

\begin{lemma}\label{lem:pwL2KG-cubic}
Let $v$ be the solution of $\Box_{t,x}v+v=\cF(t,x)$. Then one has
\begin{equation}\label{pwL2:KG-cubic}
\begin{split}
\w{t+|x|}^{0.9}|v(t,x)|&\ls\sum_{|a|\le 4}\sup_{\tau\in[0,t]}\w{\tau}^{-1/10}
\|\w{\tau+|\cdot|}\Gamma^a\cF(\tau,\cdot)\|_{L^2(\R^2)}\\
&\quad+\sum_{|a|\le 5}\|\w{\cdot}^2\Gamma^av(0,\cdot)\|_{L^2(\R^2)}
\end{split}
\end{equation}
and
\begin{equation}\label{pwL2:KG'-cubic}
\begin{split}
\w{t+|x|}|v(t,x)|&\ls\sum_{|a|\le 4}\sup_{\tau\in[0,t]}\w{\tau}^{1/10}
\|\w{\tau+|\cdot|}\Gamma^a\cF(\tau,\cdot)\|_{L^2(\R^2)}\\
&\quad+\sum_{|a|\le 5}\|\w{\cdot}^2\Gamma^av(0,\cdot)\|_{L^2(\R^2)}.
\end{split}
\end{equation}
\end{lemma}
\begin{proof}
From Theorem 1 in \cite{Georgiev92}, we know that
\begin{equation}\label{pwL2:KG1-cubic}
\begin{split}
\w{t+|x|}|v(t,x)|&\ls\sum_{j=0}^\infty\sum_{|a|\le 4}\sup_{\tau\in[0,t]}
\vp_j(\tau)\|\w{\tau+|\cdot|}\Gamma^a\cF(\tau,\cdot)\|_{L^2(\R^2)}\\
&\quad+\sum_{j=0}^\infty\sum_{|a|\le 5}\|\w{x'}\vp_j(|x'|)\Gamma^av(0,x')\|_{L^2_{x'}(\R^2)},
\end{split}
\end{equation}
where $\{\vp_j\}_{j=0}^\infty$ is a Littlewood-Paley partition of unit defined in \eqref{pwL2-phi-def}.

The remaining proof process is completely similar to that of Lemma \ref{lem:pwL2KG} except utilizing \eqref{pwL2:KG1-cubic} instead
of \eqref{pwL2:KG1}, the related details are omitted here.
\end{proof}

\begin{lemma}\label{lem:pw:nonzero-cubic}
Let $u$ be the solution of \eqref{QWE-cubic}.
Under the bootstrap assumptions \eqref{bootstrap-cubic}, for any $n\in\Z_*$ and multi-index $a$ with $|a|\le N-7$, it holds that
\begin{equation}\label{pw:nonzero-cubic}
\w{t+|x|}^{0.9}|(Z^au)_n(t,x)|\ls |n|^{|a|+5-N}(\ve+\ve_1^3)
\end{equation}
and
\begin{equation}\label{dyu:pw'-cubic}
\w{t+|x|}|(\p_yu)_n(t,x)|\ls|n|^{13-N}(\ve+\ve_1^3).
\end{equation}
\end{lemma}
\begin{proof}
Taking the Fourier transform on the both sides of \eqref{eqn:high-cubic} with respect to $y$ to obtain
\begin{equation}\label{pw:nonzero1-cubic}
(\Box_{t,x}+|n|^2)(Z^au)_n(t,x)=(\mathcal{C}^a)_n(t,x).
\end{equation}
Let $(s,z)=(|n|t,|n|x)$, $\Box_{s,z}=\p_s^2-\Delta_z$ and $(\overline{Z^a u})_n(s,z):=(Z^au)_n(\frac{s}{|n|},\frac{z}{|n|})$.
Then \eqref{pw:nonzero1-cubic} is changed to
\begin{equation*}
|n|^2(\Box_{s,z}+1)(\overline{Z^au})_n(s,z)=(\overline{\mathcal{C}^a})_n(s,z):=(\mathcal{C}^a)_n(\frac{s}{|n|},\frac{z}{|n|}).
\end{equation*}
Therefore, it is derived from \eqref{pwL2:KG-cubic} that
\begin{equation*}
\begin{split}
\w{s+|z|}^{0.9}|(\overline{Z^au})_n(s,z)|
&\ls |n|^{-2}\sum_{|a'|\le 4}\sum_{\tilde\Gamma\in\{\p_{\tau,z'},L,\Omega\}}
\sup_{\tau\in[0,s]}\w{\tau}^{-1/10}\|\w{\tau+|z'|}\tilde\Gamma^{a'}(\overline{\mathcal{C}^a})_n(\tau,z')\|_{L_{z'}^2(\R^2)}\\
&\quad+|n|^{-2}\sum_{|a'|\le 5}\sum_{\tilde\Gamma\in\{\p_{\tau,z'},L,\Omega\}}
\|\w{z'}^2\tilde\Gamma^{a'}(\overline{Z^au})_n(0,z')\|_{L_{z'}^2(\R^2)},
\end{split}
\end{equation*}
where $\tilde\Gamma'$s are the vector fields in $(\tau,z')$ variables.
As in Lemma \ref{lem:pw:nonzero}, when returning to the $(t,x)$ variables, we have
\begin{equation}\label{pw:nonzero2-cubic}
\begin{split}
\w{t+|x|}^{0.9}|(Z^au)_n(t,x)|&\ls \sum_{|a'|\le 4}
\sup_{\tau\in[0,t]}\w{\tau}^{-1/10}\|\w{\tau+|x|}\Gamma^{a'}(\mathcal{C}^a)_n(\tau, x)\|_{L^2_x(\R^2)}\\
&\quad+|n|\sum_{|a'|\le 5}\|\w{x}^2\Gamma^{a'}(Z^au)_n(0,x)\|_{L^2_x(\R^2)}.
\end{split}
\end{equation}
In addition, the definition of $\mathcal{C}^a$ in \eqref{eqn:high-cubic} leads to
\begin{equation}\label{pw:nonzero3-cubic}
\sum_{|a'|\le 4}\|\w{\tau+|x|}\Gamma^{a'}(\mathcal{C}^a)_n\|_{L_x^2}
\ls \sum_{|b|+|c|+|d|\le |a|+4}\Big(J_n^{bcd}+\sum_{m\in\Z_*,m\neq n}J_{nm}^{bcd}+
\sum_{m,l\in\Z_*, l\neq n-m}J_{nml}^{bcd}\Big),
\end{equation}
where
\begin{equation*}
\begin{split}
J_n^{bcd}&:=\|\w{\tau+|x|}(\p\p^{\le1}Z^bu)_n(P_{=0}\p\p^{\le1}Z^cu)(P_{=0}\p\p^{\le1}Z^du)\|_{L_x^2},\\
J_{nm}^{bcd}&:=\|\w{\tau+|x|}(\p\p^{\le1}Z^bu)_{n-m}(\p\p^{\le1}Z^cu)_m(P_{=0}\p\p^{\le1}Z^du)\|_{L_x^2},\\
J_{nml}^{bcd}&:=\|\w{\tau+|x|}(\p\p^{\le1}Z^bu)_{n-m-l}(\p\p^{\le1}Z^cu)_m(\p\p^{\le1}Z^du)_l\|_{L_x^2}.
\end{split}
\end{equation*}

\emph{ i) Estimate of~~$\sum_{|b|+|c|+|d|\le |a|+4}J_n^{bcd}$}

\vskip 0.2 true cm

\emph{ i-a)} For $J_n^{bcd}$ with $|b|+|c|\le |a|-5$, the assumptions \eqref{bootstrap-cubic} and the estimate \eqref{pw:zero'-cubic} imply
\begin{equation}\label{pw:nonzero-cubic-00001}
\begin{split}
\sum_{\substack{|b|+|c|+|d|\le |a|+4,\\|b|+|c|\le|a|-5}}J_n^{bcd}
&\ls\sum_{\substack{|b|+|c|+|d|\le|a|+4,\\|b|+|c|\le|a|-5}}
\|\w{\tau+|x|}(\p\p^{\le1}Z^bu)_n(P_{=0}\p\p^{\le1}Z^cu)\|_{L^\infty}\|P_{=0}\p\p^{\le1}Z^du\|_{L_x^2}\\
&\ls\sum_{|b|\le|a|-5}\ve_1^2|n|^{|b|+2+5-N}\w{\tau}^{-0.4}\ve_1\w{\tau}^{\ve_2}
\ls \ve_1^3|n|^{|a|+2-N}\w{\tau}^{-0.4+\ve_2}.
\end{split}
\end{equation}

\emph{ i-b)} For $J_n^{bcd}$ with $|b|+|c|\ge|a|-4$, we see that $|d|\le 8 \le N-12$. Then \eqref{pw:zero'-cubic} ensures that
\begin{equation}\label{pw:nonzero-cubic-1}
\begin{split}
&\sum_{\substack{|b|+|c|+|d|\le|a|+4,\\|a|-4\le|b|+|c|\le|a|+4}}J_n^{bcd}\\
&\ls\sum_{\substack{|b|+|c|+|d|\le|a|+4,\\|a|-4\le|b|+|c|\le|a|+4}}
\|\w{\tau+|x|}^{1/2}(\p\p^{\le1}Z^bu)_n(P_{=0}\p\p^{\le1}Z^cu)\|_{L_x^2}\|\w{\tau+|x|}^{1/2}P_{=0}\p\p^{\le1}Z^du\|_{L^\infty}\\
&\ls \ve_1\sum_{\substack{|a|-4\le|b|+|c|\le|a|+4}}
\|\w{\tau+|x|}^{1/2}(\p\p^{\le1}Z^bu)_n(P_{=0}\p\p^{\le1}Z^cu)\|_{L_x^2}.
\end{split}
\end{equation}
When $|b|\leq |a|-2$, it follows from \eqref{bootstrap-cubic} that
\begin{equation}\label{pw:nonzero-cubic-2}
\begin{split}
&\sum_{\substack{|a|-4\le|b|+|c|\le|a|+4,\\|b|\leq |a|-2}}
\|\w{\tau+|x|}^{1/2}(\p\p^{\le1}Z^bu)_n(P_{=0}\p\p^{\le1}Z^cu)\|_{L_x^2}\\
&\ls \sum_{\substack{|a|-4\le|b|+|c|\le|a|+4,\\|b|\leq |a|-2}}
\|\w{\tau+|x|}^{1/2}(\p\p^{\le1}Z^bu)_n\|_{L^\infty}\|P_{=0}\p\p^{\le1}Z^cu\|_{L_x^2}\\
&\ls \sum_{\substack{|b|\leq |a|-2}}\|\w{\tau+|x|}^{1/2}(\p\p^{\le1}Z^bu)_n\|_{L^\infty}E_N(\tau)\\
&\ls \sum_{\substack{|b|\leq |a|-2}}\ve_1^2|n|^{|b|+2+5-N}\w{\tau}^{-0.4+\ve_2}
\ls \ve_1^2|n|^{|a|+5-N}\w{\tau}^{-0.4+\ve_2}.
\end{split}
\end{equation}
Collecting \eqref{pw:nonzero-cubic-1} and \eqref{pw:nonzero-cubic-2} derives
\begin{equation}\label{pw:nonzero4-cubic}
\sum_{\substack{|b|+|c|+|d|\le|a|+4,\\ |b|+|c|\geq|a|-4,|b|\le|a|-2}}J_n^{bcd}
\ls \ve_1^3 |n|^{|a|+5-N}\w{\tau}^{-0.4+\ve_2}.
\end{equation}
When $|b|\geq |a|-1$, we have $|c|+1\leq 1+|a|+4-|b|\leq 6\le N-11$. Then it follows from \eqref{bootstrap'} and \eqref{pw:zero'-cubic} that
\begin{equation}\label{pw:nonzero-cubic-2-dd}
\begin{split}
&\sum_{\substack{|a|-4\le|b|+|c|\le|a|+4,\\|b|\geq |a|-1}}
\|\w{\tau+|x|}^{1/2}(\p\p^{\le1}Z^bu)_n(P_{=0}\p\p^{\le1}Z^cu)\|_{L_x^2}\\
&\ls \sum_{\substack{|a|-4\le|b|+|c|\le|a|+4,\\|b|\geq |a|-1}}
\|(\p\p^{\le1}Z^bu)_n\|_{L_x^2}\|\w{\tau+|x|}^{1/2}(P_{=0}\p\p^{\le1}Z^cu)\|_{L^\infty}\\
&\ls \ve_1^2|n|^{|b|+1-N}\w{\tau}^{\ve_2}\ls \ve_1^2|n|^{|a|+5-N}\w{\tau}^{\ve_2}.
\end{split}
\end{equation}
In view of \eqref{pw:nonzero-cubic-1} and \eqref{pw:nonzero-cubic-2-dd}, one has
\begin{equation}\label{pw:nonzero4-cubic-oth}
\sum_{\substack{|b|+|c|+|d|\le|a|+4,\\ |b|+|c|\geq|a|-4,|b|\geq|a|-1}}J_n^{bcd}
\ls \ve_1^3|n|^{|a|+5-N}\w{\tau}^{\ve_2}.
\end{equation}
Collecting \eqref{pw:nonzero-cubic-00001},\eqref{pw:nonzero4-cubic} and \eqref{pw:nonzero4-cubic-oth} yields
\begin{equation}\label{pw:nonzero-cubic-aaa}
\sum_{|b|+|c|+|d|\le|a|+4}J_n^{bcd}\ls \ve_1^3|n|^{|a|+5-N}\w{\tau}^{\ve_2}.
\end{equation}

\emph{ ii) Estimate of~~$\sum_{m\in\Z_*,m\neq n}\sum_{|b|+|c|+|d|\le |a|+4}J_{nm}^{bcd}$}

\vskip 0.2 true cm

\emph{ ii-a)} In the case of $|b|+|c|\leq |a|-2$. By \eqref{bootstrap-cubic}, we have
\begin{equation}\label{pw:nonzero-cubic-001}
\begin{split}
\sum_{\substack{|b|+|c|+|d|\le|a|+4,\\|b|+|c|\leq |a|-2}}J_{nm}^{bcd}
&\ls\sum_{\substack{|b|+|c|+|d|\le|a|+4,\\|b|+|c|\leq |a|-2}}
\|\w{\tau+|x|}(\p\p^{\le1}Z^bu)_{n-m}((\p\p^{\le1}Z^cu)_m)\|_{L^\infty}\|P_{=0}\p\p^{\le1}Z^du\|_{L_x^2}\\
&\ls\sum_{\substack{|b|+|c|\leq |a|-2}}\|\w{\tau+|x|}(\p\p^{\le1}Z^bu)_{n-m}((\p\p^{\le1}Z^cu)_m)\|_{L^\infty}E_N(\tau)\\
&\ls\sum_{\substack{|b|+|c|\leq |a|-2}}\ve_1^3|n-m|^{|b|+2+5-N}|m|^{|c|+2+5-N}\w{\tau}^{-0.8+\ve_2}\\
&\ls\ve_1^3|n-m|^{|a|+5-N}|m|^{|a|+5-N}\w{\tau}^{-0.8+\ve_2}.
\end{split}
\end{equation}

\emph{ ii-b)} In the case of $|b|+|c|\geq |a|-1$. Then $|d|\leq |a|+4-(|b|+|c|)\leq 5\leq N-12$ holds.
Hence, by \eqref{pw:zero'-cubic} we get
\begin{equation}\label{pw:nonzero-cubic-002}
\begin{split}
&\sum_{\substack{|b|+|c|+|d|\le|a|+4,\\|b|+|c|\geq |a|-1}}J_{nm}^{bcd}\\
&\ls\sum_{\substack{|b|+|c|+|d|\le|a|+4,\\|b|+|c|\geq |a|-1}}
\|\w{\tau+|x|}^{1/2}(\p\p^{\le1}Z^bu)_{n-m}((\p\p^{\le1}Z^cu)_m)\|_{L_x^2}\|\w{\tau+|x|}^{1/2}(P_{=0}\p\p^{\le1}Z^du)\|_{L^\infty}\\
&\ls\ve_1\sum_{\substack{|a|-1\leq|b|+|c|\leq |a|+4}}\underbrace{\|\w{\tau+|x|}^{1/2}(\p\p^{\le1}Z^bu)_{n-m}(\p\p^{\le1}Z^cu)_m\|_{L_x^2}}_{:=W^{bc}_{nm}(\tau)}.
\end{split}
\end{equation}

We now treat $W^{bc}_{nm}(\tau)$ in the following three cases.

\emph{ ii-b-1) Estimate of $W_{nm}(\tau)$ when $|a|-1\leq|b|+|c|\leq |a|+4, |b|\leq |a|-2$.}

We see from \eqref{bootstrap-cubic} and \eqref{bootstrap'} that
\begin{equation}\label{pw:nonzero-cubic-003}
\begin{split}
&\sum_{\substack{|a|-1\leq|b|+|c|\leq |a|+4,\\|b|\leq |a|-2}}W^{bc}_{nm}(\tau)\\
&\ls \sum_{\substack{|a|-1\leq|b|+|c|\leq |a|+4,\\|b|\leq |a|-2}}\|\w{\tau+|x|}^{1/2}(\p\p^{\le1}Z^bu)_{n-m}\|_{L^\infty}\|(\p\p^{\le1}Z^cu)_m\|_{L_x^2}\\
&\ls \sum_{\substack{|b|\leq |a|-2,|c|\leq |a|+4}}\ve_1\w{\tau}^{-0.4}|n-m|^{|b|+2+5-N}\cdot \ve_1\w{\tau}^{\ve_2}|m|^{|c|+1-N}\\
&\ls \ve_1^2\w{\tau}^{-0.4+\ve_2}|n-m|^{|a|+5-N}|m|^{|a|+5-N}.
\end{split}
\end{equation}

\emph{ ii-b-2) Estimate of $W_{nm}(\tau)$ when $|a|-1\leq|b|+|c|\leq |a|+4, |b|\geq |a|-1, |a|\geq 7$.}

~In this case, then $|c|\leq |a|+4-|b|\leq 5\leq |a|-2$ holds. We can conclude from \eqref{bootstrap-cubic} and \eqref{bootstrap'} that
\begin{equation}\label{pw:nonzero-cubic-004}
\begin{split}
&\sum_{\substack{|a|-1\leq|b|+|c|\leq |a|+4,\\|b|\geq |a|-1, |a|\geq 7}}W^{bc}_{nm}(\tau)\\
&\ls \sum_{\substack{|a|-1\leq|b|+|c|\leq |a|+4,\\|c|\leq |a|-2}}\|\w{\tau+|x|}^{1/2}(\p\p^{\le1}Z^cu)_m\|_{L^\infty}\|(\p\p^{\le1}Z^bu)_{n-m}\|_{L_x^2}\\
&\ls \sum_{\substack{|c|\leq |a|-2,|b|\leq |a|+4}}\ve_1\w{\tau}^{-0.4}|m|^{|c|+2+5-N}\cdot \ve_1\w{\tau}^{\ve_2}|n-m|^{|b|+1-N}\\
&\ls \ve_1^2\w{\tau}^{-0.4+\ve_2}|m|^{|a|+5-N}|n-m|^{|a|+5-N}.
\end{split}
\end{equation}

\emph{ ii-b-3) Estimate of $W_{nm}(\tau)$ when $|a|-1\leq|b|+|c|\leq |a|+4, |b|\geq |a|-1, |a|\leq 6$.}

Note that in this case,  $|b|,|c|\leq 10\leq N-9$ holds. Utilizing \eqref{bootstrap-cubic} and \eqref{bootstrap'} again,
one can easily obtain that
\begin{equation}\label{pw:nonzero-cubic-005}
\begin{split}
&\sum_{\substack{|a|-1\leq|b|+|c|\leq |a|+4,\\|b|\geq |a|-1, |a|\leq 6}}W^{bc}_{nm}(\tau)\\
&\ls \sum_{\substack{|b|+|c|\leq |a|+4,\\|b|,|c|\leq 10}}\mathrm{min}\{\|\w{\tau+|x|}^{1/2}(\p\p^{\le1}Z^bu)_{n-m}\|_{L^\infty}\|(\p\p^{\le1}Z^cu)_m\|_{L_x^2},\\
&\quad\quad\quad\quad\quad\quad\quad\quad\quad\|\w{\tau+|x|}^{1/2}(\p\p^{\le1}Z^cu)_m\|_{L^\infty}\|(\p\p^{\le1}Z^bu)_{n-m}\|_{L_x^2}\}\\
&\ls \sum_{\substack{|b|+|c|\leq |a|+4,\\|b|,|c|\leq 10}}\mathrm{min}\{\ve_1\w{\tau}^{-0.4}|n-m|^{|b|+2+5-N}\cdot \ve_1\w{\tau}^{\ve_2}|m|^{|c|+1-N},\\
&\quad\quad\quad\quad\quad\quad\quad\ve_1\w{\tau}^{-0.4}|m|^{|c|+2+5-N}\cdot \ve_1\w{\tau}^{\ve_2}|n-m|^{|b|+1-N}\}\\
&\ls \sum_{\substack{|b|+|c|\leq |a|+4,\\|b|,|c|\leq 10}}
\ve_1^2\w{\tau}^{-0.4+\ve_2}\mathrm{min}\{|n-m|^{|b|+7-N}|m|^{|c|+1-N},|n-m|^{|b|+1-N}|m|^{|c|+7-N}\}\\
&\ls \ve_1^2\w{\tau}^{-0.4+\ve_2}\mathrm{min}\{|n-m|^{17-N}|m|^{|a|+5-N},|n-m|^{|a|+5-N}|m|^{17-N}\}\\
&\ls \ve_1^2\w{\tau}^{-0.4+\ve_2}\mathrm{min}\{|n-m|^{-2}|m|^{|a|+5-N},|n-m|^{|a|+5-N}|m|^{-2}\}.
\end{split}
\end{equation}
By \eqref{pw:nonzero-cubic-002} and \eqref{pw:nonzero-cubic-005}, we arrive at
\begin{equation}\label{pw:nonzero-cubic-006}
\begin{split}
&\sum_{m\in\Z_*,m\neq n}\sum_{\substack{|b|+|c|+|d|\le|a|+4,\\|b|+|c|\geq |a|-1,\\|b|\geq |a|-1, |a|\leq 6}}J_{nm}^{bcd}\\
&\ls \ve_1^3\w{\tau}^{-0.4+\ve_2}\sum_{m\in\Z_*,m\neq n}\mathrm{min}\{|n-m|^{-2}|m|^{|a|+5-N},|n-m|^{|a|+5-N}|m|^{-2}\}\\
&\ls \sum_{m\in\Z_*,m\neq n,|m|\leq |n|/2}\ve_1^3\w{\tau}^{-0.4+\ve_2}|n-m|^{|a|+5-N}|m|^{-2}\\
&\quad\quad\quad+\sum_{m\in\Z_*,m\neq n,|m|\geq |n|/2}\ve_1^3\w{\tau}^{-0.4+\ve_2}|n-m|^{-2}|m|^{|a|+5-N}\\
&\ls \ve_1^3\w{\tau}^{-0.4+\ve_2}|n|^{|a|+5-N}\sum_{m\in\Z_*,m\neq n,|m|\leq |n|/2}|m|^{-2}\\
&\quad\quad+\ve_1^3\w{\tau}^{-0.4+\ve_2}|n|^{|a|+5-N}\sum_{m\in\Z_*,m\neq n,|m|\geq |n|/2}|n-m|^{-2}\\
&\ls \ve_1^3\w{\tau}^{-0.4+\ve_2}|n|^{|a|+5-N}
\end{split}
\end{equation}
Collecting \eqref{pw:nonzero-cubic-001}-\eqref{pw:nonzero-cubic-004} yields that
\begin{equation*}\label{pw:nonzero-cubic-007}
\begin{split}
&\sum_{m\in\Z_*,m\neq n}\sum_{\substack{|b|+|c|+|d|\le|a|+4,\\ 7\leq |a|\leq N-7}}J_{nm}^{bcd}\\
\ls& \sum_{m\in\Z_*,m\neq n}\ve_1^3\w{\tau}^{-0.4+\ve_2}|n-m|^{|a|+5-N}|m|^{|a|+5-N}\\
\end{split}
\end{equation*}

\begin{equation}\label{pw:nonzero-cubic-007}
\begin{split}
\ls& \sum_{m\in\Z_*,m\neq n,|m|\leq |n|/2}\ve_1^3\w{\tau}^{-0.4+\ve_2}|n-m|^{|a|+5-N}|m|^{|a|+5-N}\\
&\quad\quad\quad+\sum_{m\in\Z_*,m\neq n,|m|\geq |n|/2}\ve_1^3\w{\tau}^{-0.4+\ve_2}|n-m|^{|a|+5-N}|m|^{|a|+5-N}\\
\ls& \ve_1^3\w{\tau}^{-0.4+\ve_2}|n|^{|a|+5-N}\sum_{m\in\Z_*}|m|^{|a|+5-N}\\
&\quad\quad\quad+\ve_1^3\w{\tau}^{-0.4+\ve_2}|n|^{|a|+5-N}\sum_{m\in\Z_*,m\neq n}|n-m|^{|a|+5-N}\\
\ls& \ve_1^3\w{\tau}^{-0.4+\ve_2}|n|^{|a|+5-N}.
\end{split}
\end{equation}
Analogously, collecting \eqref{pw:nonzero-cubic-001}-\eqref{pw:nonzero-cubic-003} and \eqref{pw:nonzero-cubic-006} to get
\begin{equation}\label{pw:nonzero-cubic-008}
\sum_{m\in\Z_*,m\neq n}\sum_{\substack{|b|+|c|+|d|\le|a|+4,\\|a|\leq 6}}J_{nm}^{bcd}
\ls \ve_1^3\w{\tau}^{-0.4+\ve_2}|n|^{|a|+5-N}.
\end{equation}
By \eqref{pw:nonzero-cubic-007} and \eqref{pw:nonzero-cubic-008}, we have
\begin{equation}\label{pw:nonzero-cubic-008-add}
\sum_{m\in\Z_*,m\neq n}\sum_{|b|+|c|+|d|\le|a|+4}J_{nm}^{bcd}
\ls \ve_1^3\w{\tau}^{-0.4+\ve_2}|n|^{|a|+5-N}.
\end{equation}

\emph{ iii) Estimate of~~$\sum_{m,l\in\Z_*,l\neq n-m}\sum_{|b|+|c|+|d|\le |a|+4}J_{nml}^{bcd}$}

\vskip 0.2 true cm

\emph{ iii-a)} In the case of $|a|\geq [(N-7)/2]+4$. At this time, at least two of the three multiple indexes $b, c$ and $d$
in the above summation meet the fact that the sum of each index is less than or equal to $|a|-2$. Otherwise, $|b|+|c|+|d|\geq 2(|a|-1)\geq 2[(N-7)/2]+6>2((N-7)/2-1)+6=N-3$, which contradicts with $|b|+|c|+|d|\leq |a|+4\leq N-3$.
Without loss of generality, we can assume that $|b|\leq |a|-2$ and $|c|\leq |a|-2$. Then it is concluded
from \eqref{bootstrap-cubic} and \eqref{bootstrap'} that
\begin{equation*}\label{pw:nonzero-cubic-009}
\begin{split}
&\sum_{\substack{|b|+|c|+|d|\le|a|+4,\\|b|\leq |a|-2,|c|\leq |a|-2}}J_{nml}^{bcd}\\
&\ls\sum_{\substack{|b|+|c|+|d|\le|a|+4,\\|b|\leq |a|-2,|c|\leq |a|-2}}
\|\w{\tau+|x|}(\p\p^{\le1}Z^bu)_{n-m-l}(\p\p^{\le1}Z^cu)_m\|_{L^\infty}\|(\p\p^{\le1}Z^d u)_l\|_{L_x^2}\\
&\ls \sum_{\substack{|d|\le|a|+4,\\|b|\leq |a|-2,|c|\leq |a|-2}}
\ve_1^3|n-m-l|^{|b|+2+5-N}|m|^{|c|+2+5-N}|l|^{|d|+1-N}\w{\tau}^{-0.8+\ve_2}\\
&\ls\ve_1^3|n-m-l|^{|a|+5-N}|m|^{|a|+5-N}|l|^{|a|+5-N}\w{\tau}^{-0.8+\ve_2}.
\end{split}
\end{equation*}
Hence,
\begin{equation*}\label{pw:nonzero-cubic-0010}
\begin{split}
&\sum_{m,l\in\Z_*,l\neq n-m}\sum_{\substack{|b|+|c|+|d|\le|a|+4,\\|b|\leq |a|-2,|c|\leq |a|-2}}J_{nml}^{bcd}\\
&\ls\sum_{m,l\in\Z_*,l\neq n-m}\ve_1^3\w{\tau}^{-0.8+\ve_2}|n-m-l|^{|a|+5-N}|m|^{|a|+5-N}|l|^{|a|+5-N}\\
\end{split}
\end{equation*}

\begin{equation}\label{pw:nonzero-cubic-0010}
\begin{split}
&\ls \sum_{\substack{m,l\in\Z_*,l\neq n-m,\\|m|\leq |n|/4~\mathrm{and}~|l|\leq|n|/4}}
\ve_1^3\w{\tau}^{-0.8+\ve_2}|n-m-l|^{|a|+5-N}|m|^{|a|+5-N}|l|^{|a|+5-N}\\
&\quad\quad\quad+\sum_{\substack{m,l\in\Z_*,l\neq n-m,\\|m|\geq |n|/4~\mathrm{or}~|l|\geq |n|/4}}
\ve_1^3\w{\tau}^{-0.8+\ve_2}|n-m-l|^{|a|+5-N}|m|^{|a|+5-N}|l|^{|a|+5-N}\\
&\ls \ve_1^3\w{\tau}^{-0.8+\ve_2}|n|^{|a|+5-N}\sum_{\substack{m,l\in\Z_*}}
|m|^{|a|+5-N}|l|^{|a|+5-N}\\
&\quad\quad\quad+\ve_1^3\w{\tau}^{-0.8+\ve_2}|n|^{|a|+5-N}\sum_{\substack{m,l\in\Z_*,l\neq n-m}}
|n-m-l|^{|a|+5-N}|l|^{|a|+5-N}\\
&\ls \ve_1^3\w{\tau}^{-0.8+\ve_2}|n|^{|a|+5-N}+\ve_1^3\w{\tau}^{-0.8+\ve_2}|n|^{|a|+5-N}
\sum_{l\in\Z_*}|l|^{-2}\Big(\sum_{m\in\Z_*,m\neq n-l}|n-l-m|^{-2}\Big)\\
&\ls \ve_1^3\w{\tau}^{-0.8+\ve_2}|n|^{|a|+5-N}.
\end{split}
\end{equation}

\emph{ iii-b)} In the case of $|a|\leq [(N-7)/2]+3$. At this time, one has
that $|b|,|c|,|d|\leq |a|+4\leq [(N-7)/2]+7\leq N/2-7/2+7\leq N-9$~($\text{due to}$~$N\geq 25$).
By \eqref{bootstrap-cubic} and \eqref{bootstrap'}, we arrive at
\begin{equation*}\label{pw:nonzero-cubic-0011}
\begin{split}
&\sum_{\substack{|b|+|c|+|d|\le|a|+4,\\|a|\leq [(N-7)/2]+3}}J_{nml}^{bcd}\\
&\ls\sum_{\substack{|b|+|c|+|d|\le|a|+4,\\|b|,|c|,|d|\leq N-9}}
\mathrm{min}\big\{\|\w{\tau+|x|}(\p\p^{\le1}Z^cu)_m(\p\p^{\le1}Z^d u)_l\|_{L^\infty}\|(\p\p^{\le1}Z^bu)_{n-m-l}\|_{L_x^2},\\
&\quad\quad\quad\quad\quad\quad\|\w{\tau+|x|}(\p\p^{\le1}Z^bu)_{n-m-l}(\p\p^{\le1}Z^d u)_l\|_{L^\infty}\|(\p\p^{\le1}Z^cu)_m\|_{L_x^2},\\
&\quad\quad\quad\quad\quad\quad\|\w{\tau+|x|}(\p\p^{\le1}Z^bu)_{n-m-l}(\p\p^{\le1}Z^cu)_m\|_{L^\infty}\|(\p\p^{\le1}Z^d u)_l\|_{L_x^2}\big\}\\
&\ls \sum_{\substack{|b|+|c|+|d|\le|a|+4,\\|b|,|c|,|d|\leq N-9}}
\mathrm{min}\big\{\ve_1^3|m|^{|c|+2+5-N}|l|^{|d|+2+5-N}|n-m-l|^{|b|+1-N}\w{\tau}^{-0.8+\ve_2},\\
&\quad\quad\quad\quad\quad\quad\ve_1^3|n-m-l|^{|b|+2+5-N}|l|^{|d|+2+5-N}|m|^{|c|+1-N}\w{\tau}^{-0.8+\ve_2},\\
&\quad\quad\quad\quad\quad\quad\ve_1^3|n-m-l|^{|b|+2+5-N}|m|^{|c|+2+5-N}|l|^{|d|+1-N}\w{\tau}^{-0.8+\ve_2}\}\\
&\ls \ve_1^3\w{\tau}^{-0.8+\ve_2}\mathrm{min}\{|m|^{-2}|l|^{-2}|n-m-l|^{|a|+5-N},|n-m-l|^{-2}|l|^{-2}|m|^{|a|+5-N},\\
&\quad\quad\quad\quad\quad\quad\quad\quad\quad|n-m-l|^{-2}|m|^{-2}|l|^{|a|+5-N}\big\}.
\end{split}
\end{equation*}
Hence,
\begin{equation*}\label{pw:nonzero-cubic-0012}
\begin{split}
&\sum_{m,l\in\Z_*,l\neq n-m}\sum_{\substack{|b|+|c|+|d|\le|a|+4,\\|a|\leq [(N-7)/2]+3}}J_{nml}^{bcd}\\
&\ls \ve_1^3\w{\tau}^{-0.8+\ve_2}\mathrm{min}\big\{|m|^{-2}|l|^{-2}|n-m-l|^{|a|+5-N},|n-m-l|^{-2}|l|^{-2}|m|^{|a|+5-N},\\
&\quad\quad\quad\quad\quad\quad\quad\quad\quad|n-m-l|^{-2}|m|^{-2}|l|^{|a|+5-N}\}\\
\end{split}
\end{equation*}

\begin{equation}\label{pw:nonzero-cubic-0012}
\begin{split}
&\ls \sum_{\substack{m,l\in\Z_*,l\neq n-m,\\|m|\leq |n|/4~\mathrm{and}~|l|\leq|n|/4}}
\ve_1^3\w{\tau}^{-0.8+\ve_2}|m|^{-2}|l|^{-2}|n-m-l|^{|a|+5-N}\\
&\quad\quad\quad+\sum_{\substack{m,l\in\Z_*,l\neq n-m,\\|m|\geq |n|/4~\mathrm{or}~|l|\geq |n|/4}}
\ve_1^3\w{\tau}^{-0.8+\ve_2}\mathrm{min}\{|n-m-l|^{-2}|l|^{-2}|m|^{|a|+5-N},\\
&\quad\quad\quad\quad\quad\quad\quad\quad\quad\quad\quad\quad\quad\quad\quad\quad\quad
\quad\quad|n-m-l|^{-2}|m|^{-2}|l|^{|a|+5-N}\big\}\\
&\ls \ve_1^3\w{\tau}^{-0.8+\ve_2}|n|^{|a|+5-N}\sum_{\substack{m,l\in\Z_*}}|m|^{-2}|l|^{-2}\\
&\quad\quad\quad+\ve_1^3\w{\tau}^{-0.8+\ve_2}|n|^{|a|+5-N}\sum_{\substack{m,l\in\Z_*,l\neq n-m}}|n-m-l|^{-2}|l|^{-2}\\
&\ls \ve_1^3\w{\tau}^{-0.8+\ve_2}|n|^{|a|+5-N}+\ve_1^3\w{\tau}^{-0.8+\ve_2}|n|^{|a|+5-N}\\
&\quad\quad\quad\cdot\sum_{l\in\Z_*}|l|^{-2}\Big(\sum_{m\in\Z_*,m\neq n-l}|n-l-m|^{-2}\Big)\\
&\ls \ve_1^3\w{\tau}^{-0.8+\ve_2}|n|^{|a|+5-N}.
\end{split}
\end{equation}
By \eqref{pw:nonzero-cubic-0010}-\eqref{pw:nonzero-cubic-0012}, one has
\begin{equation}\label{pw:nonzero-cubic-0013}
\sum_{m,l\in\Z_*,l\neq n-m}\sum_{|b|+|c|+|d|\le|a|+4}J_{nml}^{bcd}\ls \ve_1^3\w{\tau}^{-0.8+\ve_2}|n|^{|a|+5-N}.
\end{equation}
Inserting \eqref{pw:nonzero-cubic-aaa}, \eqref{pw:nonzero-cubic-008-add} and \eqref{pw:nonzero-cubic-0013} into \eqref{pw:nonzero3-cubic}
derives
\begin{equation}\label{pw:nonzero3-cubic-final}
\sum_{|a'|\le 4}\|\w{\tau+|x|}\Gamma^{a'}(\mathcal{C}^a)_n\|_{L_x^2}\ls \ve_1^3\w{\tau}^{\ve_2}|n|^{|a|+5-N}.
\end{equation}
For the second line of \eqref{pw:nonzero2-cubic}, \eqref{initial:data} with $d=2$ leads to
\begin{equation}\label{pw:nonzero10-cubic}
|n|\sum_{|a'|\le 5}\|\w{x}^2\Gamma^{a'}(Z^au)_n(0,x)\|_{L^2_x(\R^2)}\ls \ve |n|^{|a|+5-N}.
\end{equation}
Substituting \eqref{pw:nonzero3-cubic-final}-\eqref{pw:nonzero10-cubic} into \eqref{pw:nonzero2-cubic} yields \eqref{pw:nonzero-cubic}.

Finally, we give the proof of \eqref{dyu:pw'-cubic}.
Let $Z^a=\p_y$ and repeat the steps to obtain \eqref{pw:nonzero2-cubic} with \eqref{pwL2:KG'-cubic} instead of \eqref{pwL2:KG-cubic}.
Then one has
\begin{equation}\label{pw:nonzero11-cubic-11}
\begin{split}
\w{t+|x|}|(\p_y u)_n(t,x)|&\ls |n|^{1/10}\sum_{|a'|\le 4}
\sup_{\tau\in[0,t]}\w{\tau}^{1/10}\|\w{\tau+|x|}\Gamma^{a'}(\p_y \mathcal{C}(\p u,\p^2u))_n(\tau,x)\|_{L^2_x(\R^2)}\\
&\quad+|n|\sum_{|a'|\le 5}\|\w{x}^2\Gamma^{a'}(\p_y u)_n(0,x)\|_{L^2_x(\R^2)}\\
&\ls |n|^{1/10}\sum_{|a'|\le 4}
\sup_{\tau\in[0,t]}\w{\tau}^{1/10}\|\w{\tau+|x|}\Gamma^{a'}(\p_y \mathcal{C}(\p u,\p^2u))_n(\tau,x)\|_{L^2_x(\R^2)}\\
&\quad+\ve|n|^{6-N},
\end{split}
\end{equation}
where  \eqref{pw:nonzero10-cubic} for $|a|=1$ has been used.

In addition, analogously to the estimate \eqref{pw:nonzero3-cubic}, we have
\begin{equation}\label{pw:nonzero3-cubic-12}
\sum_{|a'|\le 4}\|\w{\tau+|x|}\Gamma^{a'}(\p_y \mathcal{C}(\p u,\p^2u))_n\|_{L_x^2}
\ls \sum_{|b|+|c|+|d|\le 5}\Big(J_n^{bcd}+\sum_{m\in\Z_*,m\neq n}J_{nm}^{bcd}+
\sum_{m,l\in\Z_*, l\neq n-m}J_{nml}^{bcd}\Big).
\end{equation}
By \eqref{bootstrap-cubic} and \eqref{pw:zero'-cubic}, one has
\begin{equation}\label{pw:nonzero3-cubic-13}
\begin{split}
\sum_{|b|+|c|+|d|\le 5}J_n^{bcd}
&\ls\sum_{\substack{|b|+|c|+|d|\le 5}}
\|\w{\tau+|x|}(\p\p^{\le1}Z^bu)_n(P_{=0}\p\p^{\le1}Z^cu)\|_{L^\infty}\|P_{=0}\p\p^{\le1}Z^du\|_{L_x^2}\\
&\ls\sum_{|b|\le 5}\ve_1^2|n|^{|b|+2+5-N}\w{\tau}^{-0.4} E_N(\tau)\ls\sum_{|b|\le 5}\ve_1^2|n|^{|b|+2+5-N}\w{\tau}^{-0.4}\ve_1\w{\tau}^{\ve_2}\\
&\ls \ve_1^3|n|^{12-N}\w{\tau}^{-0.4+\ve_2}.
\end{split}
\end{equation}
By the bootstrap assumptions in \eqref{bootstrap-cubic}, we obtain
\begin{equation*}\label{pw:nonzero3-cubic-14}
\begin{split}
&\sum_{m\in\Z_*,m\neq n}\sum_{|b|+|c|+|d|\le 5}J_{nm}^{bcd}\\
&\ls\sum_{m\in\Z_*,m\neq n}\sum_{|b|+|c|+|d|\le 5}
\|\w{\tau+|x|}(\p\p^{\le1}Z^bu)_{n-m}((\p\p^{\le1}Z^cu)_m)\|_{L^\infty}\|P_{=0}\p\p^{\le1}Z^du\|_{L_x^2}\\
\end{split}
\end{equation*}

\begin{equation}\label{pw:nonzero3-cubic-14}
\begin{split}
&\ls\sum_{m\in\Z_*,m\neq n}\sum_{|b|+|c|+|d|\le 5}\|\w{\tau+|x|}(\p\p^{\le1}Z^bu)_{n-m}((\p\p^{\le1}Z^cu)_m)\|_{L^\infty} E_N(\tau)\\
&\ls\ve_1^3\sum_{m\in\Z_*,m\neq n}\sum_{|b|+|c|\le 5}|n-m|^{|b|+2+5-N}|m|^{|c|+2+5-N}\w{\tau}^{-0.8+\ve_2}\\
&\ls\w{\tau}^{-0.8+\ve_2}\ve_1^3\sum_{m\in\Z_*,m\neq n}|n-m|^{12-N}|m|^{12-N}\\
&\ls\ve_1^3\w{\tau}^{-0.8+\ve_2}|n|^{12-N}.
\end{split}
\end{equation}
Analogously,
\begin{equation}\label{pw:nonzero3-cubic-15}
\begin{split}
&\sum_{m,l\in\Z_*, l\neq n-m}\sum_{|b|+|c|+|d|\le 5}J_{nml}^{bcd}\\
&\ls\sum_{m,l\in\Z_*, l\neq n-m}\sum_{|b|+|c|+|d|\le 5}
\|\w{\tau+|x|}(\p\p^{\le1}Z^bu)_{n-m-l}(\p\p^{\le1}Z^cu)_m\|_{L^\infty}\|(\p\p^{\le1}Z^d u)_l\|_{L_x^2}\\
&\ls \sum_{m,l\in\Z_*, l\neq n-m}\sum_{|b|+|c|+|d|\le 5}
\ve_1^3|n-m-l|^{|b|+2+5-N}|m|^{|c|+2+5-N}|l|^{|d|+1-N}\w{\tau}^{-0.8+\ve_2}\\
&\ls\ve_1^3\w{\tau}^{-0.8+\ve_2}\sum_{m,l\in\Z_*, l\neq n-m}|n-m-l|^{12-N}|m|^{12-N}|l|^{12-N}\\
&\ls\ve_1^3\w{\tau}^{-0.8+\ve_2}|n|^{12-N}.
\end{split}
\end{equation}
Collecting \eqref{pw:nonzero3-cubic-12}-\eqref{pw:nonzero3-cubic-15} derives
\begin{equation*}
\sum_{|a'|\le 4}\|\w{\tau+|x|}\Gamma^{a'}(\p_y \mathcal{C}(\p u,\p^2u))_n\|_{L_x^2}
\ls \ve_1^3|n|^{12-N}\w{\tau}^{-0.4+\ve_2}.
\end{equation*}
This, together with \eqref{pw:nonzero11-cubic-11}, yields \eqref{dyu:pw'-cubic}.
\end{proof}

\begin{corollary}
Let $u$ be the solution to \eqref{QWE-cubic}.
Under the bootstrap assumptions \eqref{bootstrap-cubic}, for any multi-index $a$ with $|a|\le N-7$, it holds that
\begin{equation}\label{pw:nonzero'-cubic}
|P_{\neq0}Z^au(t,x,y)|\ls \w{t+|x|}^{-0.9}(\ve+\ve_1^3).
\end{equation}
\end{corollary}
\begin{proof}
For $|a|\le N-7$, it follows from \eqref{pw:nonzero-cubic} that
\begin{equation*}
|P_{\neq0}Z^au(t,x,y)|=\Big|\sum_{n\in\Z_*}e^{{\rm i}ny}(Z^au)_n(t,x)\Big|\ls\w{t+|x|}^{-0.9}(\ve+\ve_1^3)\sum_{n\in\Z_*}|n|^{-2},
\end{equation*}
which yields \eqref{pw:nonzero'-cubic} immediately.
\end{proof}

\section{Energy estimates}\label{sect5}
\begin{lemma}
For smooth function $f(t,x), x\in \R^d$ with $d=2,3$, it holds that
\begin{equation}\label{wL2:ineq}
\begin{split}
\|\w{t-|x|}\p_{t,x}^2f\|_{L^2(\R^d)}&\ls\sum_{|a|\le1}\|\p_{t,x}\Gamma^af\|_{L^2(\R^d)}
+\|\w{t+|x|}\Box_{t,x}f\|_{L^2(\R^d)},\\
\|\w{t+|x|}\bar\p\p_{t,x}f\|_{L^2(\R^d:|x|\ge \w{t}/3)}&\ls\sum_{|a|\le1}\|\p_{t,x}\Gamma^af\|_{L^2(\R^d)}
+\|\w{t-|x|}\p_{t,x}^2f\|_{L^2(\R^d)}.
\end{split}
\end{equation}
\end{lemma}
\begin{proof}
We only give the proof of \eqref{wL2:ineq} with $d=3$ since one can analogously check
that \eqref{wL2:ineq} holds for $d=2$.
Direct computation derives that
\begin{equation}\label{wL2:ineq1}
\begin{split}
(t^2-|x|^2)\p_t^2f&=t^2(\p_t^2-\Delta_x)f+t(t\p_i+x_i\p_t)\p_if-tx_i\p_t\p_if\\
&\quad-x_i(x_i\p_t+t\p_i)\p_tf+tx_i\p_i\p_tf\\
&=t^2\Box_{t,x}f+tL_i\p_if-x_iL_i\p_tf.
\end{split}
\end{equation}
Similarly, we can achieve that for $i=1,2,3$
\begin{equation}\label{wL2:ineq2}
\begin{split}
(t^2-|x|^2)\p_i\p_tf&=t(t\p_i+x_i\p_t)\p_tf-tx_i\p_t^2f-x_j(x_j\p_t+t\p_j)\p_if\\
&\quad+t(x_j\p_i-x_i\p_j)\p_jf+tx_i\Delta_xf\\
&=tL_i\p_tf-tx_i\Box_{t,x}f-x_jL_j\p_if+t\Omega_{ji}\p_jf,\\
(t^2-|x|^2)\Delta_xf&=-(t^2-|x|^2)\Box_{t,x}f+(t^2-|x|^2)\p_t^2f,
\end{split}
\end{equation}
where  the summations from $j=1$ to $3$ in \eqref{wL2:ineq2} are omitted.
On the other hand, it follows from the integration by parts that
\begin{equation}\label{wL2:ineq3}
\|\w{t-|x|}\nabla_x^2f\|_{L_x^2}\ls\|\w{t-|x|}\Delta_xf\|_{L_x^2}+\|\p_xf\|_{L_x^2}.
\end{equation}
Collecting \eqref{wL2:ineq1}-\eqref{wL2:ineq3} ensures the first inequality in \eqref{wL2:ineq}.
In addition, one has
\begin{equation*}
\begin{split}
|x|\bar\p_i\p_{t,x}f&=|x|(\frac{x_i}{|x|}\p_t+\p_i)\p_{t,x}f=x_i\p_t\p_{t,x}f+|x|\p_i\p_{t,x}f\\
&=(x_i\p_t+t\p_i)\p_{t,x}f+(|x|-t)\p_i\p_{t,x}f=L_i\p_{t,x}f+(|x|-t)\p_i\p_{t,x}f.
\end{split}
\end{equation*}
This leads to the second inequality in \eqref{wL2:ineq}.
\end{proof}
\begin{remark}
The first weighted inequality in \eqref{wL2:ineq} contains the hyperbolic rotations $L_i$ instead of the scaling operator $t\p_t+r\p_r$
used in \cite{KS96}.  This is a one of the interesting observations in our paper.
\end{remark}

\subsection{Energy estimates on $\R^3\times\T$}

The following lemma illustrates that the auxiliary energy $\cX_k(t)$ can be controlled by the standard energy $E_k(t)$
with a small correction.

\begin{lemma}\label{lem:aux:energy}
Under the bootstrap assumptions \eqref{bootstrap}, for any integer $1\le k\le N$, it holds that
\begin{equation}\label{aux:energy}
\cX_k(t)\ls E_k(t)+\ve_1^2.
\end{equation}
\end{lemma}
\begin{proof}
It can be concluded from \eqref{aux:energy:def}, \eqref{eqn:high} and \eqref{wL2:ineq} with $d=3$ that
\begin{equation}\label{aux:energy1}
\cX_k(t)\ls E_k(t)+\sum_{|a|\le k-1}\|\w{t+|x|}P_{=0}Q^a\|_{L^2_x}
\ls E_k(t)+\sum_{|b|+|c|\le k-1}\Big(\cJ^{bc}+\sum_{n\in\Z_*}\cJ_n^{bc}\Big),
\end{equation}
where
\begin{equation*}
\begin{split}
\cJ^{bc}&:=\|\w{t+|x|}(P_{=0}\p\p^{\le1}Z^bu)(P_{=0}\p\p^{\le1}Z^cu)\|_{L^2_x},\\
\cJ_n^{bc}&:=\|\w{t+|x|}(\p\p^{\le1}Z^bu)_n(\p\p^{\le1}Z^cu)_{-n}\|_{L_x^2}.
\end{split}
\end{equation*}
For any $|b|+|c|\le k-1\le N-1$, we can deduce $|b|\le N-13$ or $|c|\le N-13$.
Thus \eqref{pw:zero'} leads to
\begin{equation}\label{aux:energy2}
\begin{split}
\sum_{\substack{|b|+|c|\le k-1,\\|b|\le N-13}}\cJ^{bc}
&\ls\sum_{\substack{|b|+|c|\le k-1,\\|b|\le N-13}}
\|\w{t+|x|}P_{=0}\p\p^{\le1}Z^bu\|_{L^\infty}\|P_{=0}\p\p^{\le1}Z^cu\|_{L^2_x}\\
&\ls\ve_1 E_k(t),\\
\sum_{\substack{|b|+|c|\le k-1,\\|c|\le N-13}}\cJ^{bc}
&\ls\sum_{\substack{|b|+|c|\le k-1,\\|c|\le N-13}}
\|P_{=0}\p\p^{\le1}Z^bu\|_{L^2_x}\|\w{t+|x|}P_{=0}\p\p^{\le1}Z^cu\|_{L^\infty}\\
&\ls\ve_1 E_k(t).
\end{split}
\end{equation}
On the other hand, \eqref{bootstrap}, \eqref{bootstrap'} and $N\ge25$ with $|b|\le N-10$ or $|c|\le N-10$ imply
\begin{equation}\label{aux:energy3}
\sum_{n\in\Z_*}\cJ_n^{bc}\ls\ve_1^2\sum_{n\in\Z_*}|n|^{|b|+|c|+10-2N}
\ls\ve_1^2\sum_{n\in\Z_*}|n|^{9-N}\ls\ve_1^2.
\end{equation}
Substituting \eqref{aux:energy2} and \eqref{aux:energy3} into \eqref{aux:energy1} yields \eqref{aux:energy}.
\end{proof}

We next prove the highest order energy estimate for the solution $u$ of \eqref{QWE}.

\begin{lemma}\label{lem:energy:high}
Under the bootstrap assumptions \eqref{bootstrap}, it holds that
\begin{equation}\label{energy:high}
E_N^2(t)\ls\ve^2+(\ve+\ve_1)\int_0^t\w{s}^{-1}E_N^2(s)ds.
\end{equation}
\end{lemma}
\begin{proof}
For any multi-index $a$, multiplying \eqref{eqn:high} by $\p_tZ^au$ yields that
\begin{equation}\label{energy:high1}
\begin{split}
&\quad\frac12\p_t|\p Z^au|^2-\sum_{i=1}^4\p_i(\p_tZ^au\p_iZ^au)=\p_tZ^auQ^a\\
&=\p_tZ^au\p^2_{\alpha\beta}Z^au\cQ^{\alpha\beta}
+\p_tZ^au\sum_{\substack{b+c\le a,\\|b|,|c|<|a|}}
F_{abc}^{\alpha\beta\gamma\delta}\p^2_{\alpha\beta}Z^bu\p^2_{\gamma\delta}Z^cu\\
&\quad+\p_tZ^au\sum_{\substack{b+c\le a,\\|b|<|a|}}
Q_{abc}^{\alpha\beta\gamma}\p^2_{\alpha\beta}Z^bu\p_\gamma Z^cu
+\p_tZ^au\sum_{b+c\le a}S_{abc}^{\alpha\beta}\p_\alpha Z^bu\p_\beta Z^cu,\\
&\cQ^{\alpha\beta}:=2F^{\alpha\beta\gamma\delta}\p^2_{\gamma\delta}u
+Q^{\alpha\beta\gamma}\p_\gamma u,
\end{split}
\end{equation}
where the summations over $\alpha,\beta,\gamma,\delta=0,\cdots,4$ in \eqref{energy:high1} are omitted.
By the symmetric conditions $F^{\alpha\beta\gamma\delta}=F^{\beta\alpha\gamma\delta}$ and $Q^{\alpha\beta\gamma}=Q^{\beta\alpha\gamma}$,
one has $\cQ^{\alpha\beta}=\cQ^{\beta\alpha}$.
For the first term on the second line of \eqref{energy:high1}, we have
\begin{equation}\label{energy:high2}
\begin{split}
&\quad\p_tZ^au\p^2_{\alpha\beta}Z^au\cQ^{\alpha\beta}\\
&=\p_\alpha(\p_tZ^au\p_\beta Z^au\cQ^{\alpha\beta})
-\p_tZ^au\p_\beta Z^au\p_\alpha\cQ^{\alpha\beta}
-\p_t\p_\alpha Z^au\p_\beta Z^au\cQ^{\alpha\beta}\\
&=\p_\alpha(\p_tZ^au\p_\beta Z^au\cQ^{\alpha\beta})
-\p_tZ^au\p_\beta Z^au\p_\alpha\cQ^{\alpha\beta}
-\frac12\p_t(\p_\alpha Z^au\p_\beta Z^au\cQ^{\alpha\beta})\\
&\quad+\frac12\p_\alpha Z^au\p_\beta Z^au\p_t\cQ^{\alpha\beta}.
\end{split}
\end{equation}
On the other hand, it can be deduced from \eqref{pw:zero'} and \eqref{pw:nonzero'} that
\begin{equation}\label{energy:high3}
|\p^{\le1}\cQ^{\alpha\beta}(t,x,y)|+|\p Z^{a'}u(t,x,y)|\ls\w{t}^{-1}(\ve+\ve_1),\quad|a'|\le N-12.
\end{equation}
Integrating \eqref{energy:high1} and \eqref{energy:high2} for all $|a|\le N$ over $[0,t]\times\R^3\times\T$ yields
\begin{equation}\label{energy:high4}
\begin{split}
E_N^2(t)&\ls E_N^2(0)+(\ve+\ve_1)(E_N^2(0)+E_N^2(t))+(\ve+\ve_1)\int_0^t\w{s}^{-1}E_N^2(s)ds\\
&\quad+\int_0^tE_N(s)\Big\{\sum_{\substack{|b|+|c|\le N,\\|b|,|c|<N}}\|\p^2Z^bu\p^2Z^cu\|_{L^2_{x,y}}
+\sum_{\substack{|b|+|c|\le N,\\|b|<N}}\|\p^2Z^bu\p Z^cu\|_{L^2_{x,y}}\Big\}ds\\
&\quad+\sum_{|b|+|c|\le N}\int_0^tE_N(s)\|\p Z^bu\p Z^cu\|_{L^2_{x,y}}ds,
\end{split}
\end{equation}
where we have used \eqref{energy:high3}.
Due to $N\ge25$, then $|b|\le N-13$ or $|c|\le N-13$ holds.
Applying \eqref{energy:high3} again derives
\begin{equation}\label{energy:high5}
\begin{split}
&\sum_{\substack{|b|+|c|\le N,\\|b|,|c|<N}}\|\p^2Z^bu\p^2Z^cu\|_{L^2_{x,y}}
+\sum_{\substack{|b|+|c|\le N,\\|b|<N}}\|\p^2Z^bu\p Z^cu\|_{L^2_{x,y}}
+\sum_{|b|+|c|\le N}\|\p Z^bu\p Z^cu\|_{L^2_{x,y}}\\
&\ls\w{s}^{-1}(\ve+\ve_1)E_N(s).
\end{split}
\end{equation}
Inserting \eqref{initial:data} and \eqref{energy:high5} into \eqref{energy:high4} with the smallness of $\ve_0$ and $\ve_1$ implies \eqref{energy:high}.
\end{proof}

With the help of the partial null condition, the uniform lower order energy estimate can be established as follows.
\begin{lemma}\label{lem:energy}
Suppose that \eqref{null:def} holds.
Under the bootstrap assumptions \eqref{bootstrap}, it holds that
\begin{equation}\label{energy}
E_{N-10}(t)\ls\ve+\ve_1^2.
\end{equation}
\end{lemma}
\begin{proof}
Integrating \eqref{energy:high1} for all $|a|\le N-10$ over $[0,t]\times\R^3\times\T$ implies that
\begin{equation}\label{energy1}
E_{N-10}^2(t)\ls E_{N-10}^2(0)+\sum_{|a|\le N-10}\int_0^tE_{N-10}(s)\|Q^a\|_{L^2_{x,y}}ds.
\end{equation}
By virtue of the facts that ${\rm Id}=P_{=0}+P_{\neq0}$ and $P_{=0}\p_y=0$, it follows from \eqref{eqn:high} that
\begin{equation}\label{energy2}
\begin{split}
Q^a&=\sum_{b+c\le a}(Q^{bc}_1+Q^{bc}_2+Q^{bc}_3),\\
Q^{bc}_1&:=\sum_{\alpha,\beta,\gamma,\delta=0}^4F_{abc}^{\alpha\beta\gamma\delta}
(P_{\neq0}\p^2_{\alpha\beta}Z^bu)\p^2_{\gamma\delta}Z^cu
+\sum_{\alpha,\beta,\gamma=0}^4Q_{abc}^{\alpha\beta\gamma}
(P_{\neq0}\p^2_{\alpha\beta}Z^bu)\p_\gamma Z^cu\\
&\quad+\sum_{\alpha,\beta=0}^4S_{abc}^{\alpha\beta}(P_{\neq0}\p_\alpha Z^bu)\p_\beta Z^cu,\\
Q^{bc}_2&:=\sum_{\alpha,\beta,\gamma,\delta=0}^4F_{abc}^{\alpha\beta\gamma\delta}
(P_{=0}\p^2_{\alpha\beta}Z^bu)(P_{\neq0}\p^2_{\gamma\delta}Z^cu)
+\sum_{\alpha,\beta,\gamma=0}^4Q_{abc}^{\alpha\beta\gamma}
(P_{=0}\p^2_{\alpha\beta}Z^bu)(P_{\neq0}\p_\gamma Z^cu)\\
&\quad+\sum_{\alpha,\beta=0}^4S_{abc}^{\alpha\beta}(P_{=0}\p_\alpha Z^bu)(P_{\neq0}\p_\beta Z^cu),\\
Q^{bc}_3&:=\sum_{\alpha,\beta,\gamma,\delta=0}^3F_{abc}^{\alpha\beta\gamma\delta}
(P_{=0}\p^2_{\alpha\beta}Z^bu)(P_{=0}\p^2_{\gamma\delta}Z^cu)
+\sum_{\alpha,\beta,\gamma=0}^3Q_{abc}^{\alpha\beta\gamma}
(P_{=0}\p^2_{\alpha\beta}Z^bu)(P_{=0}\p_\gamma Z^cu)\\
&\quad+\sum_{\alpha,\beta=0}^3S_{abc}^{\alpha\beta}(P_{=0}\p_\alpha Z^bu)(P_{=0}\p_\beta Z^cu).
\end{split}
\end{equation}
By $|b|+|c|\le|a|\le N-10$, \eqref{pw:nonzero'} leads to
\begin{equation}\label{energy3}
\begin{split}
\|Q^{bc}_1\|_{L^2_{x,y}}&\ls\|P_{\neq0}\p\p^{\le1}Z^bu\|_{L^\infty}
\|\p\p^{\le1}Z^cu\|_{L^2_{x,y}}\\
&\ls\w{s}^{-1.4}(\ve+\ve_1)E_{N-9}(s)\ls\w{s}^{-1.3}\ve_1(\ve+\ve_1),\\
\|Q^{bc}_2\|_{L^2_{x,y}}&\ls\|P_{=0}\p\p^{\le1}Z^bu\|_{L^2_{x,y}}
\|P_{\neq0}\p\p^{\le1}Z^cu\|_{L^\infty}\\
&\ls\w{s}^{-1.4}(\ve+\ve_1)E_{N-9}(s)\ls\w{s}^{-1.3}\ve_1(\ve+\ve_1).
\end{split}
\end{equation}
In the region $|x|\le\w{s}/2$, it can be concluded from \eqref{pw:zero} and Hardy inequality on $\R^3$ that
\begin{equation}\label{energy4}
\begin{split}
\|Q^{bc}_3\|_{L^2_{x,y}(|x|\le\frac{\w{s}}{2})}&\ls\w{s}^{-1}(\|\w{s-|x|}P_{=0}\p^2Z^bu\|_{L^2_{x,y}}
+\|\w{x}^{-1}\w{s-|x|}P_{=0}\p Z^bu\|_{L^2_{x,y}})\\
&\qquad\times\|\w{x}P_{=0}\p\p^{\le1}Z^cu\|_{L^\infty(|x|\le\frac{\w{s}}{2})}\\
&\ls\w{s}^{-1}(\|\w{s-|x|}P_{=0}\p^2Z^bu\|_{L^2_{x,y}}+\|P_{=0}\p Z^bu\|_{L^2_{x,y}}))\|\w{x}P_{=0}\p\p^{\le1}Z^cu\|_{L^\infty(|x|\le\frac{\w{s}}{2})}\\
&\ls\w{s}^{-1.4}\ve_1(E_{N-10}(s)+\cX_{N-9}(s))\ls\w{s}^{-1.3}\ve_1^2.
\end{split}
\end{equation}
In the region $|x|\ge\w{s}/2$, applying \eqref{null:structure} to $Q^{bc}_3$ with \eqref{null:high} shows
\begin{equation*}\label{energy5}
\begin{split}
|Q^{bc}_3|&\ls|P_{=0}\bar\p\p^{\le1}Z^bu||P_{=0}\p\p^{\le1}Z^cu|
+|P_{=0}\p\p^{\le1}Z^bu||P_{=0}\bar\p\p^{\le1}Z^cu|.
\end{split}
\end{equation*}
This, together with \eqref{pw:good}, yields
\begin{equation}\label{energy5}
\begin{split}
\|Q^{bc}_3\|_{L^2_{x,y}(|x|\ge\frac{\w{s}}{2})}
&\ls\|P_{=0}\bar\p\p^{\le1}Z^bu\|_{L^\infty(|x|\ge\frac{\w{s}}{2})}
\|P_{=0}\p\p^{\le1}Z^cu\|_{L^2_{x,y}}\\
&\quad+\|P_{=0}\p\p^{\le1}Z^bu\|_{L^2_{x,y}}
\|P_{=0}\bar\p\p^{\le1}Z^cu\|_{L^\infty(|x|\ge\frac{\w{s}}{2})}\\
&\ls\w{s}^{-1.4}\ve_1E_{N-9}(s)\ls\w{s}^{-1.3}\ve_1^2.
\end{split}
\end{equation}
By collecting \eqref{energy1}--\eqref{energy5}, for any $0\leq t_1\le t$, we arrive at
\begin{equation*}
\begin{split}
E_{N-10}^2(t_1)&\ls E_{N-10}^2(0)+\ve_1(\ve+\ve_1)\int_0^{t_1} E_{N-10}(s)\w{s}^{-1.3}ds\\
&\ls E_{N-10}^2(0)+\ve_1(\ve+\ve_1)\sup_{s\in[0,t]}E_{N-10}(s).
\end{split}
\end{equation*}
Hence,
\begin{equation*}
\Big(\sup_{s\in[0,t]}E_{N-10}(s)\Big)^2\ls E_{N-10}^2(0)+\ve_1(\ve+\ve_1)\sup_{s\in[0,t]}E_{N-10}(s).
\end{equation*}
This, together with \eqref{initial:data} ensures that
\begin{equation*}
E_{N-10}(t)\ls\sup_{s\in[0,t]}E_{N-10}(s)\ls E_{N-10}(0)+\ve_1(\ve+\ve_1)\ls\ve+\ve_1^2,
\end{equation*}
which implies \eqref{energy}.
\end{proof}

\subsection{Energy estimates on $\R^2\times\T$}

\begin{lemma}\label{lem:aux:energy-cubic}
Under the bootstrap assumptions \eqref{bootstrap-cubic}, for any integer $1\le k\le N$, it holds that
\begin{equation}\label{aux:energy-cubic}
\cX_k(t)\ls E_k(t)+\ve_1^3.
\end{equation}
\end{lemma}
\begin{proof}
It can be concluded from the definition of $\cX_k(t)$, \eqref{eqn:high-cubic} and \eqref{wL2:ineq} with $d=2$ that
\begin{equation}\label{aux:energy1-cubic}
\begin{split}
\cX_k(t)\ls& E_k(t)+\sum_{|a|\le k-1}\|\w{t+|x|}P_{=0}\mathcal{C}^a\|_{L^2_x}\\
\ls& E_k(t)+\sum_{|b|+|c|+|d|\le k-1}\Big(\cJ^{bcd}+\sum_{n\in\Z_*}\cJ_n^{bcd}+\sum_{n,l\in\Z_*,n+l\neq0}\cJ_{nl}^{bcd}\Big),
\end{split}
\end{equation}
where
\begin{equation}\label{aux:energy1-cubic-1}
\begin{split}
\cJ^{bcd}&:=\|\w{t+|x|}(P_{=0}\p\p^{\le1}Z^bu)(P_{=0}\p\p^{\le1}Z^cu)(P_{=0}\p\p^{\le1}Z^du)\|_{L^2_x},\\
\cJ_n^{bcd}&:=\|\w{t+|x|}(P_{=0}\p\p^{\le1}Z^bu)(\p\p^{\le1}Z^cu)_n(\p\p^{\le1}Z^du)_{-n}\|_{L_x^2},\\
\cJ_{nl}^{bcd}&:=\|\w{t+|x|}(\p\p^{\le1}Z^bu)_n(\p\p^{\le1}Z^cu)_{l}(\p\p^{\le1}Z^du)_{-(n+l)}\|_{L_x^2}.
\end{split}
\end{equation}

Next we estimate the second term on the right-hand side of \eqref{aux:energy1-cubic}.

Because of $|b|+|c|+|d|\le k-1\leq N-1$, we see that at least two of the three multiple indexes $b, c$ and $d$
in the summation of \eqref{aux:energy1-cubic} meet the fact that the sum of each index is less than or equal to $N-12$
(otherwise, $|b|+|c|+|d|\geq 2(N-11)\geq N+N-22\geq N+3 $, which contradicts with $|b|+|c|+|d|\leq N-1$).
In view of the above fact, at least one of the following three cases always holds:
\begin{equation}\label{aux:energy1-cubic-2}
\begin{split}
A_1)~|b|\leq N-12, \quad |c|\leq N-12,\\
A_2)~|b|\leq N-12, \quad |d|\leq N-12,\\
A_3)~|c|\leq N-12, \quad |d|\leq N-12.
\end{split}
\end{equation}

\emph{i) Estimate of~$\sum_{|b|+|c|+|d|\le k-1}\cJ^{bcd}$.}

\vskip 0.1 true cm

By symmetry, we only need to consider one of the three cases in \eqref{aux:energy1-cubic-2}. Without loss of generality, we assume
that $|b|\leq N-12, |c|\leq N-12$. It follows from \eqref{pw:zero'-cubic} and the definition of $E_k(t)$ that
\begin{equation}\label{aux:energy1-cubic-3}
\begin{split}
&\sum_{\substack{|b|+|c|+|d|\le k-1,\\|b|,|c|\le N-12}}\cJ^{bcd}\\
\ls& \sum_{\substack{|b|+|c|+|d|\le k-1,\\|b|,|c|\le N-12}}
\|\w{t+|x|}P_{=0}\p\p^{\le1}Z^bu\cdot P_{=0}\p\p^{\le1}Z^cu\|_{L^\infty}\|P_{=0}\p\p^{\le1}Z^du\|_{L^2_x}\\
\ls& \ve_1^2 E_k(t).
\end{split}
\end{equation}

\emph{ii) Estimate of~$\sum_{n\in\Z_*}\sum_{|b|+|c|+|d|\le k-1}\cJ_n^{bcd}$.}

\vskip 0.1 true cm

By symmetry of the indexes $c,d$ in $\cJ_n^{bcd}$, for case \emph{$A_1$)} and case \emph{$A_2$)} in \eqref{aux:energy1-cubic-2},
one only requires to estimate one of them. In this case, it suffices to divide the estimates of $\sum_{n\in\Z_*}\sum_{|b|+|c|+|d|\le k-1}\cJ_n^{bcd}$ into the following two cases.

\vskip 0.1 true cm

\emph{ii-A) $|b|\leq N-12, |c|\leq N-12$.}

\vskip 0.1 true cm

By \eqref{bootstrap-cubic}, \eqref{bootstrap'} and \eqref{pw:zero'-cubic}, we obtain
\begin{equation}\label{aux:energy1-cubic-4}
\begin{split}
&\sum_{n\in\Z_*}\sum_{\substack{|b|+|c|+|d|\le k-1,\\|b|,|c|\le N-12}}\cJ_{n}^{bcd}\\
\ls& \sum_{n\in\Z_*}\sum_{\substack{|b|+|c|+|d|\le k-1,\\|b|,|c|\le N-12}}
\|\w{t+|x|}P_{=0}\p\p^{\le1}Z^bu(\p\p^{\le1}Z^cu)_n\|_{L^\infty}\|(\p\p^{\le1}Z^du)_{-n}\|_{L^2_x}\\
\ls& \sum_{n\in\Z_*}\sum_{|c|+|d|\le k-1}\ve_1^2 \w{t}^{-0.4}|n|^{|c|+2+5-N}\cdot \ve_1 \w{t}^{\varepsilon_2}|n|^{|d|+1-N}\\
\ls&\ve_1^3 \w{t}^{-0.4+\varepsilon_2}\sum_{n\in\Z_*}\sum_{|c|+|d|\le N-1}|n|^{|c|+|d|+8-2N}\\
\ls& \ve_1^3 \w{t}^{-0.4+\varepsilon_2}\sum_{n\in\Z_*}|n|^{7-N}\ls \ve_1^3 \w{t}^{-0.4+\ve_2}.
\end{split}
\end{equation}

\vskip 0.1 true cm

\emph{ii-B) $|c|\leq N-12, |d|\leq N-12$.}

\vskip 0.1 true cm

It follows from \eqref{bootstrap-cubic} and the definition of $E_k(t)$ that
\begin{equation}\label{aux:energy1-cubic-5}
\begin{split}
&\sum_{n\in\Z_*}\sum_{\substack{|b|+|c|+|d|\le k-1,\\|c|,|d|\le N-12}}\cJ_{n}^{bcd}\\
\ls& \sum_{n\in\Z_*}\sum_{\substack{|b|+|c|+|d|\le k-1,\\|c|,|d|\le N-12}}
\|\w{t+|x|}(\p\p^{\le1}Z^cu)_n(\p\p^{\le1}Z^du)_{-n}\|_{L^\infty}\|P_{=0}\p\p^{\le1}Z^bu\|_{L^2_x}\\
\ls& \sum_{n\in\Z_*}\sum_{|c|+|d|\le k-1}\ve_1^2 \w{t}^{-0.8}|n|^{|c|+2+5-N}|n|^{|d|+2+5-N}E_k(t)\\
\ls&\ve_1^2 \w{t}^{-0.8}\sum_{n\in\Z_*}\sum_{|c|+|d|\le N-1}|n|^{|c|+|d|+14-2N} E_k(t)\\
\ls& \ve_1^2 \w{t}^{-0.8}E_k(t)\sum_{n\in\Z_*}|n|^{13-N}\ls \ve_1^2 \w{t}^{-0.8}E_k(t).
\end{split}
\end{equation}

In view of \eqref{aux:energy1-cubic-4}-\eqref{aux:energy1-cubic-5}, we have
\begin{equation}\label{aux:energy1-cubic-6}
\sum_{n\in\Z_*}\sum_{|b|+|c|+|d|\le k-1}\cJ_{n}^{bcd}\ls \ve_1^2 E_k(t)+\ve_1^3.
\end{equation}

\emph{iii) Estimate of~$\sum_{n,l\in\Z_*,n+l\neq0}\sum_{|b|+|c|+|d|\le k-1}\cJ_{nl}^{bcd}$.}

\vskip 0.1 true cm

By symmetry of the indexes $b, c$ in $\cJ_{nl}^{bcd}$, for case \emph{$A_2)$} and case \emph{$A_3$)} in \eqref{aux:energy1-cubic-2},
we only need to estimate one of them. Then the estimates of $\sum_{n\in\Z_*}\sum_{|b|+|c|+|d|\le k-1}\cJ_n^{bcd}$
can be divided into the following two cases.

\vskip 0.1 true cm

\emph{iii-A) $|b|\leq N-12, |c|\leq N-12$.}

\vskip 0.1 true cm

By \eqref{bootstrap-cubic} and \eqref{bootstrap'}, we obtain
\begin{equation}\label{aux:energy1-cubic-7}
\begin{split}
&\sum_{n,l\in\Z_*,n+l\neq0}\sum_{\substack{|b|+|c|+|d|\le k-1,\\|b|,|c|\le N-12}}\cJ_{nl}^{bcd}\\
\ls& \sum_{n,l\in\Z_*,n+l\neq0}\sum_{\substack{|b|+|c|+|d|\le k-1,\\|b|,|c|\le N-12}}
\|\w{t+|x|}(\p\p^{\le1}Z^bu)_n(\p\p^{\le1}Z^cu)_l\|_{L^\infty}\|(\p\p^{\le1}Z^du)_{-(n+l)}\|_{L^2_x}\\
\ls& \sum_{n,l\in\Z_*,n+l\neq0}\sum_{\substack{|b|+|c|+|d|\le k-1,\\|b|,|c|\le N-12}}\ve_1^2 \w{t}^{-0.8}|n|^{|b|+2+5-N}
|l|^{|c|+2+5-N}\cdot \ve_1 \w{t}^{\varepsilon_2}|n+l|^{|d|+1-N}\\
\ls& \ve_1^3 \w{t}^{-0.8+\varepsilon_2}\sum_{n,l\in\Z_*,n+l\neq0}|n|^{-2}|l|^{-2}
\ls \ve_1^3 \w{t}^{-0.8+\varepsilon_2}.
\end{split}
\end{equation}

\emph{iii-B) $|c|\leq N-12, |d|\leq N-12$.}

\vskip 0.1 true cm

Utilizing \eqref{bootstrap-cubic} and \eqref{bootstrap'} again to obtain
\begin{equation}\label{aux:energy1-cubic-8}
\begin{split}
&\sum_{n,l\in\Z_*,n+l\neq0}\sum_{\substack{|b|+|c|+|d|\le k-1,\\|c|,|d|\le N-12}}\cJ_{nl}^{bcd}\\
\ls& \sum_{n,l\in\Z_*,n+l\neq0}\sum_{\substack{|b|+|c|+|d|\le k-1,\\|c|,|d|\le N-12}}
\|\w{t+|x|}(\p\p^{\le1}Z^cu)_l(\p\p^{\le1}Z^du)_{-(n+l)}\|_{L^\infty}\|(\p\p^{\le1}Z^bu)_n\|_{L^2_x}\\
\ls& \sum_{n,l\in\Z_*,n+l\neq0}\sum_{\substack{|b|+|c|+|d|\le k-1,\\|c|,|d|\le N-12}}\ve_1^2 \w{t}^{-0.8}|l|^{|c|+2+5-N}
|n+l|^{|d|+2+5-N}\cdot \ve_1 \w{t}^{\varepsilon_2}|n|^{|b|+1-N}\\
\ls& \ve_1^3 \w{t}^{-0.8+\varepsilon_2}\sum_{n,l\in\Z_*,n+l\neq0}|l|^{-2}|n+l|^{-2}
\ls \ve_1^3 \w{t}^{-0.8+\varepsilon_2}.
\end{split}
\end{equation}

By \eqref{aux:energy1-cubic-7}-\eqref{aux:energy1-cubic-8}, one has
\begin{equation}\label{aux:energy1-cubic-9}
\sum_{n,l\in\Z_*,n+l\neq0}\sum_{|b|+|c|+|d|\le k-1}\cJ_{nl}^{bcd}\ls \ve_1^3.
\end{equation}

Substituting \eqref{aux:energy1-cubic-3}, \eqref{aux:energy1-cubic-6} and \eqref{aux:energy1-cubic-9}
into \eqref{aux:energy1-cubic} yields \eqref{aux:energy-cubic}.
\end{proof}

\begin{lemma}\label{lem:energy:high-cubic}
Under the bootstrap assumptions \eqref{bootstrap-cubic}, it holds that
\begin{equation}\label{energy:high-cubic}
E_N^2(t)\ls\ve^2+(\ve+\ve_1)^2E_N^2(t)+(\ve+\ve_1)^2\int_0^t\w{s}^{-1}E_N^2(s)ds.
\end{equation}
\end{lemma}
\begin{proof}
For any multi-index $a$, multiplying \eqref{eqn:high-cubic} by $\p_tZ^au$ yields that
\begin{equation}\label{energy:high1-cubic}
\begin{split}
&\quad\frac12\p_t|\p Z^au|^2-\sum_{i=1}^3\p_i(\p_tZ^au\p_iZ^au)=\p_tZ^au \mathcal{C}^a\\
&=\p_tZ^au\p^2_{\alpha\beta}Z^au\cR^{\alpha\beta}
+\p_tZ^au\sum_{\substack{b+c+d\le a,\\|b|,|c|,|d|<|a|}}
F_{abcd}^{\alpha\beta\gamma\delta\mu\nu}\p^2_{\alpha\beta}Z^bu\p^2_{\gamma\delta}Z^cu\p^2_{\mu\nu}Z^du\\
&\quad+\p_tZ^au\sum_{\substack{b+c+d\le a,\\|b|,|c|<|a|}}
G_{abcd}^{\alpha\beta\gamma\delta\mu}\p^2_{\alpha\beta}Z^bu\p_{\gamma\delta} Z^cu\p_{\mu} Z^du\\
&\quad+\p_tZ^au\sum_{\substack{b+c+d\le a,\\|b|<|a|}}
H_{abcd}^{\alpha\beta\gamma\delta}\p^2_{\alpha\beta}Z^bu\p_{\gamma} Z^cu\p_{\delta} Z^du\\
&\quad+\p_tZ^au\sum_{b+c+d\le a}S_{abcd}^{\alpha\beta\gamma}\p_\alpha Z^bu\p_\beta Z^cuu\p_\gamma Z^du,\\
&\cR^{\alpha\beta}:=3F^{\alpha\beta\gamma\delta\mu\nu}\p^2_{\gamma\delta}u\p^2_{\mu\nu}u
+2G^{\alpha\beta\gamma\delta\mu}\p_{\gamma\delta} u\p_\mu u+H^{\alpha\beta\gamma\delta}\p_{\gamma} u\p_\delta u,
\end{split}
\end{equation}
where the summations of $\alpha,\beta,\gamma,\delta,\mu,\nu=0,\cdots,3$ for repeated indexes in \eqref{energy:high1-cubic}
are omitted.
By the symmetric conditions $F^{\alpha\beta\gamma\delta\mu\nu}=F^{\beta\alpha\gamma\delta\mu\nu}$, $G^{\alpha\beta\gamma\delta\mu}=G^{\beta\alpha\gamma\delta\mu}$ and
$H^{\alpha\beta\gamma\delta}=H^{\beta\alpha\gamma\delta}$, we see that $\cR^{\alpha\beta}=\cR^{\beta\alpha}$.
For the first term on the second line of \eqref{energy:high1-cubic}, we have
\begin{equation}\label{energy:high2-cubic}
\begin{split}
&\quad\p_tZ^au\p^2_{\alpha\beta}Z^au\cR^{\alpha\beta}\\
&=\p_\alpha(\p_tZ^au\p_\beta Z^au\cR^{\alpha\beta})
-\p_tZ^au\p_\beta Z^au\p_\alpha\cR^{\alpha\beta}
-\p_t\p_\alpha Z^au\p_\beta Z^au\cR^{\alpha\beta}\\
&=\p_\alpha(\p_tZ^au\p_\beta Z^au\cR^{\alpha\beta})
-\p_tZ^au\p_\beta Z^au\p_\alpha\cR^{\alpha\beta}
-\frac12\p_t(\p_\alpha Z^au\p_\beta Z^au\cR^{\alpha\beta})\\
&\quad+\frac12\p_\alpha Z^au\p_\beta Z^au\p_t\cR^{\alpha\beta}.
\end{split}
\end{equation}
In addition, it can be deduced from \eqref{pw:zero'-cubic} and \eqref{pw:nonzero'-cubic} that
\begin{subequations}\label{energy:high3-cubic}\begin{align}
|\p^{\le1}\cR^{\alpha\beta}(t,x,y)|&\ls \w{t}^{-1}(\ve+\ve_1)^2,\label{energy:high3-cubic-A}\\
|\p\p^{\le1}Z^{a'}u(t,x,y)\p\p^{\le1}Z^{b'}u(t,x,y)|&\ls \w{t}^{-1}(\ve+\ve_1)^2,\nonumber\\
\mathrm{for}~|a'|,|b'|&\leq N-12.\label{energy:high3-cubic-B}
\end{align}
\end{subequations}
Integrating \eqref{energy:high1-cubic} and \eqref{energy:high2-cubic} for all $|a|\le N$ over $[0,t]\times\R^2\times\T$ leads to
\begin{equation*}\label{energy:high4-cubic}
\begin{split}
E_N^2(t)&\ls E_N^2(0)+(\ve+\ve_1)^2(E_N^2(0)+E_N^2(t))+(\ve+\ve_1)^2\int_0^t\w{s}^{-1}E_N^2(s)ds\\
&\quad+\int_0^tE_N(s)\Big\{\sum_{\substack{|b|+|c|+|d|\le N,\\|b|,|c|,|d|<N}}\|\p^2Z^bu\p^2Z^cu\p^2Z^du\|_{L^2_{x,y}}\\
\end{split}
\end{equation*}

\begin{equation}\label{energy:high4-cubic}
\begin{split}
&\quad+\sum_{\substack{|b|+|c|+|d|\le N,\\|b|,|c|<N}}\|\p^2Z^bu\p^2Z^cu\p Z^du\|_{L^2_{x,y}}\\
&\quad+\sum_{\substack{|b|+|c|+|d|\le N,\\|b|<N}}\|\p^2Z^bu\p Z^cu\p Z^du\|_{L^2_{x,y}}\\
&\quad+\sum_{|b|+|c|+|d|\le N}\|\p Z^bu\p Z^cu\p Z^du\|_{L^2_{x,y}}\Big\}ds,
\end{split}
\end{equation}
where \eqref{energy:high3-cubic-A} is used.

Due to $|b|+|c|+|d|\leq N$, we see that at least two of the three multiple indexes $b, c$ and $d$
meet the fact that the sum of each index is less than or equal to $N-12$ (otherwise, $|b|+|c|+|d|\geq 2(N-11)= N+N-22\geq N+3$,
which contradicts with $|b|+|c|+|d|\leq N$). Hence we can apply the estimate \eqref{energy:high3-cubic-B} to deduce that
\begin{equation}\label{energy:high5-cubic}
\begin{split}
&\sum_{\substack{|b|+|c|+|d|\le N,\\|b|,|c|,|d|<N}}\|\p^2Z^bu\p^2Z^cu\p^2Z^du\|_{L^2_{x,y}}
+\sum_{\substack{|b|+|c|+|d|\le N,\\|b|,|c|<N}}\|\p^2Z^bu\p^2Z^cu\p Z^du\|_{L^2_{x,y}}\\
&\quad+\sum_{\substack{|b|+|c|+|d|\le N,\\|b|<N}}\|\p^2Z^bu\p Z^cu\p Z^du\|_{L^2_{x,y}}+\sum_{|b|+|c|+|d|\le N}\|\p Z^bu\p Z^cu\p Z^du\|_{L^2_{x,y}}\\
&\ls\w{s}^{-1}(\ve+\ve_1)^2 E_N(s).
\end{split}
\end{equation}
Inserting \eqref{initial:data} and \eqref{energy:high5-cubic} into \eqref{energy:high4-cubic} derives \eqref{energy:high-cubic}.
\end{proof}

\begin{lemma}\label{lem:energy-cubic}
Suppose that \eqref{null:def-cubic} holds. Under the bootstrap assumptions \eqref{bootstrap-cubic}, it holds that
\begin{equation}\label{energy-cubic}
E_{N-9}(t)\ls \ve+\ve_1(\ve+\ve_1)^2.
\end{equation}
\end{lemma}
\begin{proof}
Integrating \eqref{energy:high1-cubic} for all $|a|\le N-9$ over $[0,t]\times\R^2\times\T$ yields
\begin{equation}\label{energy1-cubic}
E_{N-9}^2(t)\ls E_{N-9}^2(0)+\sum_{|a|\le N-9}\int_0^tE_{N-9}(s)\|\mathcal{C}^a\|_{L^2_{x,y}}ds.
\end{equation}
By virtue of the decomposition ${\rm Id}=P_{=0}+P_{\neq0}$ and $P_{=0}\p_y=0$, from \eqref{eqn:high-cubic} one can obtain
\begin{equation}\label{energy2-cubic}
\|\mathcal{C}^a\|_{L^2_{x,y}}
\ls \sum_{b+c+d\le a}\Big(\|Q^{bcd}_1\|_{L^2_{x,y}}+\|Q^{bcd}_2\|_{L^2_{x,y}}+\|Q^{bcd}_3\|_{L^2_{x,y}}+\|Q^{bcd}_4\|_{L^2_{x,y}}\Big),
\end{equation}
where
\begin{equation*}
\begin{split}
Q^{bcd}_1&:=(P_{\neq0}\p\p^{\le1}Z^bu)(P_{=0}\p\p^{\le1}Z^cu)(P_{=0}\p\p^{\le1}Z^du),\\
Q^{bcd}_2&:=(P_{\neq0}\p\p^{\le1}Z^bu)(P_{\neq0}\p\p^{\le1}Z^cu)(P_{=0}\p\p^{\le1}Z^du),\\
Q^{bcd}_3&:=(P_{\neq0}\p\p^{\le1}Z^bu)(P_{\neq0}\p\p^{\le1}Z^cu)(P_{\neq0}\p\p^{\le1}Z^du),\\
\end{split}
\end{equation*}

\begin{equation*}
\begin{split}
Q^{bcd}_4&:=\sum_{\alpha,\beta,\gamma,\delta,\mu,\nu=0}^3F_{abcd}^{\alpha\beta\gamma\delta\mu\nu}
(P_{=0}\p^2_{\alpha\beta}Z^bu)(P_{=0}\p^2_{\gamma\delta}Z^cu)(P_{=0}\p^2_{\mu\nu}Z^du)\\
&\quad+\sum_{\alpha,\beta,\gamma,\delta,\mu=0}^3G_{abcd}^{\alpha\beta\gamma\delta\mu}
(P_{=0}\p^2_{\alpha\beta}Z^bu)(P_{=0}\p_{\gamma\delta} Z^cu)(P_{=0}\p_\mu Z^du)\\
&\quad+\sum_{\alpha,\beta,\gamma,\delta=0}^3H_{abcd}^{\alpha\beta\gamma\delta}
(P_{=0}\p^2_{\alpha\beta}Z^bu)(P_{=0}\p_\gamma Z^cu)(P_{=0}\p_\delta Z^du)\\
&\quad+\sum_{\alpha,\beta=0}^3S_{abcd}^{\alpha\beta\gamma}(P_{=0}\p_\alpha Z^bu)(P_{=0}\p_\beta Z^cu)(P_{=0}\p_\gamma Z^du).
\end{split}
\end{equation*}
Due to $|b|\le N-9$, \eqref{pw:nonzero'-cubic} leads to
\begin{equation}\label{energy3-cubic}
\begin{split}
&\sum_{b+c+d\le a}(\|Q^{bcd}_1\|_{L^2_{x,y}}+\|Q^{bcd}_2\|_{L^2_{x,y}}+\|Q^{bcd}_3\|_{L^2_{x,y}})\\
&\ls \sum_{|b|\leq N-9,|c|+|d|\le |a|}\|P_{\neq0}\p\p^{\le1}Z^bu\|_{L^\infty}
(\|(P_{=0}\p\p^{\le1}Z^cu)(P_{=0}\p\p^{\le1}Z^du)\|_{L^2_{x,y}}\\
&\quad+\|(P_{\neq0}\p\p^{\le1}Z^cu)(P_{=0}\p\p^{\le1}Z^du)\|_{L^2_{x,y}}+\|(P_{\neq0}\p\p^{\le1}Z^cu)(P_{\neq0}\p\p^{\le1}Z^du)\|_{L^2_{x,y}})\\
&\ls (\ve+\ve_1)\w{s}^{-0.9}\sum_{|c|+|d|\le |a|}
(\|(P_{=0}\p\p^{\le1}Z^cu)(P_{=0}\p\p^{\le1}Z^du)\|_{L^2_{x,y}}\\
&\quad+\|(P_{\neq0}\p\p^{\le1}Z^cu)(P_{=0}\p\p^{\le1}Z^du)\|_{L^2_{x,y}}+\|(P_{\neq0}\p\p^{\le1}Z^cu)(P_{\neq0}\p\p^{\le1}Z^du)\|_{L^2_{x,y}}).
\end{split}
\end{equation}
Note that $|c|+|d|\le |a|\leq N-9$ holds. Then $|c|\leq N-12$ or $|d|\leq N-12$ (otherwise,
$|c|+|d|\geq 2(N-11)= N+N-22\geq N+3$, which contradicts with $|c|+|d|\leq N-9$). Hence we can use
\eqref{bootstrap-cubic}, Lemma \ref{lem:proj}, \eqref{pw:zero'-cubic} and \eqref{pw:nonzero'-cubic} to derive that
\begin{equation}\label{energy4-cubic}
\begin{split}
&\sum_{|c|+|d|\le |a|}
(\|(P_{=0}\p\p^{\le1}Z^cu)(P_{=0}\p\p^{\le1}Z^du)\|_{L^2_{x,y}}+\|(P_{\neq0}\p\p^{\le1}Z^cu)(P_{=0}\p\p^{\le1}Z^du)\|_{L^2_{x,y}}\\
&\quad+\|(P_{\neq0}\p\p^{\le1}Z^cu)(P_{\neq0}\p\p^{\le1}Z^du)\|_{L^2_{x,y}})\\
&\ls (\ve+\ve_1)\w{s}^{-1/2}E_{N-8}(s)\ls \ve_1(\ve+\ve_1)\w{s}^{-1/2+\ve_2}.
\end{split}
\end{equation}
Substituting \eqref{energy4-cubic} into \eqref{energy3-cubic}, one obtains
\begin{equation}\label{energy5-cubic}
\sum_{b+c+d\le a}(\|Q^{bcd}_1\|_{L^2_{x,y}}+\|Q^{bcd}_2\|_{L^2_{x,y}}+\|Q^{bcd}_3\|_{L^2_{x,y}})
\ls \ve_1(\ve+\ve_1)^2\w{s}^{-1.4+\ve_2}.
\end{equation}

Next, we deal with $\sum_{b+c+d\le a}\|Q^{bcd}_4\|_{L^2_{x,y}}$ by exploiting the partial null condition
structure of the nonlinearity in \eqref{null:def-cubic}.
In the region $|x|\le\w{s}/2$, it can be deduced from \eqref{bootstrap-cubic}, \eqref{pw:zero-cubic}, \eqref{pw:awaycone-cubic} and \eqref{energy4-cubic} that
\begin{equation*}\label{energy6-cubic}
\begin{split}
&\sum_{b+c+d\le a}\|Q^{bcd}_4\|_{L^2_{x,y}(|x|\leq \w{s}/2)}\\
&\ls \sum_{b+c+d\le a}\|P_{=0}\p\p^{\le1}Z^bu\|_{L^\infty(|x|\leq \w{s}/2)}\|(P_{=0}\p\p^{\le1}Z^cu)(P_{=0}\p\p^{\le1}Z^du)\|_{L^2_{x,y}}\\
\end{split}
\end{equation*}

\begin{equation}\label{energy6-cubic}
\begin{split}
&\ls \sum_{|b|\le N-9}(\|P_{=0}\p\p^{\le1}Z^bu\|_{L^\infty(\w{s}/3\leq|x|\leq \w{s}/2)}+\|P_{=0}\p\p^{\le1}Z^bu\|_{L^\infty(|x|\leq \w{s}/3)})\\
&\quad\quad\quad\quad\times\ve_1(\ve+\ve_1)\w{s}^{-1/2+\ve_2}\\
&\ls \sum_{|b|\le N-9}\Big(\w{s}^{-1}(E_{|b|+3}(s)+\cX_{|b|+3}(s))+\w{s}^{-1}\mathrm{ln}^{\frac{1}{2}}(e+s)(E_{|b|+3}(s)+\cX_{|b|+3}(s))\Big)\\
&\quad\quad\quad\quad\times\ve_1(\ve+\ve_1)\w{s}^{-1/2+\ve_2}\\
&\ls \ve_1^2(\ve+\ve_1)\w{s}^{-3/2+2\ve_2}\mathrm{ln}^{\frac{1}{2}}(e+s).
\end{split}
\end{equation}
In the region $|x|\ge\w{s}/2$, applying \eqref{null:structure-cubic} to $Q^{bcd}_4$ with \eqref{null:high-cubic} shows
\begin{equation*}
\sum_{b+c+d\le a}|Q^{bcd}_4|\ls \sum_{b+c+d\le a}|P_{=0}\bar\p\p^{\le1}Z^bu||P_{=0}\p\p^{\le1}Z^cu||P_{=0}\p\p^{\le1}Z^du|.
\end{equation*}
This, together with \eqref{pw:good-cubic} and \eqref{energy4-cubic}, yields
\begin{equation}\label{energy7-cubic}
\begin{split}
&\sum_{b+c+d\le a}\|Q^{bcd}_4\|_{L^2_{x,y}(|x|\ge \w{s}/2)}\\
&\ls \sum_{b+c+d\le a}\|P_{=0}\bar\p\p^{\le1}Z^bu\|_{L^\infty(|x|\ge \w{s}/2)}\|(P_{=0}\p\p^{\le1}Z^cu)(P_{=0}\p\p^{\le1}Z^du)\|_{L^2_{x,y}}\\
&\ls \sum_{|b|\le N-9}\w{s}^{-1}(E_{|b|+3}(s)+\cX_{|b|+3}(s))\cdot\ve_1(\ve+\ve_1)\w{s}^{-1/2+\ve_2}\\
&\ls \ve_1^2(\ve+\ve_1)\w{s}^{-3/2+2\ve_2}.
\end{split}
\end{equation}
By inserting \eqref{energy5-cubic}-\eqref{energy7-cubic} into \eqref{energy2-cubic} and then substituting the resulting
estimate of \eqref{energy2-cubic} into \eqref{energy1-cubic}, for any $0\leq t_1\le t$, we arrive at
\begin{equation*}
\begin{split}
E_{N-9}^2(t_1)&\ls E_{N-9}^2(0)+\ve_1(\ve+\ve_1)^2\int_0^{t_1} E_{N-9}(s)\w{s}^{-1.4+2\ve_2}\mathrm{ln}^{\frac{1}{2}}(e+s)ds\\
&\ls E_{N-9}^2(0)+\ve_1(\ve+\ve_1)^2\int_0^{t_1} E_{N-9}(s)\w{s}^{-1.1}ds\\
&\ls E_{N-9}^2(0)+\ve_1(\ve+\ve_1)^2\sup_{s\in[0,t]}E_{N-9}(s).
\end{split}
\end{equation*}
Hence,
\begin{equation*}
\Big(\sup_{s\in[0,t]}E_{N-9}(s)\Big)^2\ls E_{N-9}^2(0)+\ve_1(\ve+\ve_1)^2\sup_{s\in[0,t]}E_{N-9}(s).
\end{equation*}
This together with \eqref{initial:data} ensures that
\begin{equation*}
E_{N-9}(t)\ls\sup_{s\in[0,t]}E_{N-9}(s)\ls E_{N-9}(0)+\ve_1(\ve+\ve_1)^2\ls \ve+\ve_1(\ve+\ve_1)^2,
\end{equation*}
which implies \eqref{energy-cubic}. This completes the proof of Lemma \ref{lem:energy-cubic}.
\end{proof}

\section{Proofs of the theorems}\label{sect6}
\subsection{Proof of Theorem \ref{thm1}}
\begin{proof}[Proof of $(A_1)$ in Theorem \ref{thm1}]
From Lemmas \ref{lem:pw:nonzero}, \ref{lem:aux:energy}, \ref{lem:energy:high}, \ref{lem:energy} and Gronwall's inequality,
there are constants $C_1,C_2\ge1$ such that
\begin{equation*}
\begin{split}
E_N(t)+\cX_N(t)&\le C_1\ve(1+t)^{C_2(\ve+\ve_1)}+C_1\ve_1^2,\\
E_{N-10}(t)+\cX_{N-10}(t)&\le C_1\ve+C_1\ve_1^2,\\
\w{t+|x|}^{1.4}|(Z^au)_n(t,x)|&\le C_1|n|^{|a|+7-N}(\ve+\ve_1^2),\quad|a|\le N-8,n\in\Z_*.
\end{split}
\end{equation*}
Choosing $\ve_0=\min\{\frac{1}{16C_1^2},\frac{1}{10C_2(1+4C_1)}\}$, $\ve_1=4C_1\ve$, $\ve_2=C_2(1+4C_1)\ve\le1/10$, then for $\ve\le\ve_0$ we can obtain
\begin{equation*}
\begin{split}
E_N(t)+\cX_N(t)&\le\frac14\ve_1(1+t)^{\ve_2}+4C_1^2\ve_0\ve_1\le\frac12\ve_1(1+t)^{\ve_2},\\
E_{N-10}(t)+\cX_{N-10}(t)&\le\frac14\ve_1+\frac14\ve_1\le\frac12\ve_1,\\
\w{t+|x|}^{1.4}|(Z^au)_n(t,x)|&\le\frac12\ve_1|n|^{|a|+7-N},\quad|a|\le N-8,n\in\Z_*.
\end{split}
\end{equation*}
This, together with the local existence of classical solution to \eqref{QWE}
(see \cite{Hormander97book})
yields that \eqref{QWE} with \eqref{null:def} admits a unique global solution $u\in \bigcap\limits_{j=0}^{N+1}C^{j}([0,\infty), H^{N+1-j}(\R^3\times\T)))$.
Moreover, \eqref{dyu:pw} can be achieved by \eqref{pw:zero'}, \eqref{pw:nonzero} and \eqref{dyu:pw'}.
\end{proof}

\begin{proof}[Proof of $(A_2)$ in Theorem \ref{thm1}]
According to Lemmas \ref{lem:pw:nonzero}, \ref{lem:aux:energy}, \ref{lem:energy:high} and Gronwall's inequality,
there are constants $C_3,C_4\ge1$ such that
\begin{equation*}
\begin{split}
E_N(t)+\cX_N(t)&\le C_3\ve(1+t)^{C_4(\ve+\ve_1)}+C_3\ve_1^2,\\
\w{t+|x|}^{1.4}|(Z^au)_n(t,x)|&\le C_3|n|^{|a|+7-N}(\ve+\ve_1^2),\quad|a|\le N-8,n\in\Z_*.
\end{split}
\end{equation*}
Let $\kappa_0=\frac{1}{C_4(1+4eC_3)}$, $\ve_0=\frac{1}{16eC_3^2}$, $\ve_1=4eC_3\ve$, $\ve_2=0$.
Then for $\ve\le\ve_0$ and $t\le T_\ve=e^{\kappa_0/\ve}-1$ we can see that
\begin{equation*}
\begin{split}
E_N(t)+\cX_N(t)&\le eC_3\ve+4eC_3^2\ve_0\ve_1\le\frac12\ve_1,\\
\w{t+|x|}^{1.4}|(Z^au)_n(t,x)|&\le\frac12\ve_1|n|^{|a|+7-N},\quad|a|\le N-8,n\in\Z_*.
\end{split}
\end{equation*}
This, together with the local existence of classical solution to \eqref{QWE} ensures that \eqref{QWE} admits a unique solution
$u\in \bigcap\limits_{j=0}^{N+1}C^{j}([0,T_\ve], H^{N+1-j}(\R^3\times\T)))$. In addition, \eqref{dyu:pw} comes
from \eqref{pw:zero'}, \eqref{pw:nonzero} and \eqref{dyu:pw'}.
\end{proof}

\subsection{Proofs of Theorem \ref{thm1-cubic}}\label{sect8}
\begin{proof}[Proof of $(B_1)$ in Theorem \ref{thm1-cubic}]
From Lemmas \ref{lem:pw:nonzero-cubic}, \ref{lem:aux:energy-cubic}, \ref{lem:energy:high-cubic}, \ref{lem:energy-cubic},
there is a constant $C_5\ge1$ such that
\begin{equation}\label{proof-of-theorem-cubic}
\begin{split}
&E_N^2(t)\le C_5\ve^2+C_5(\ve+\ve_1)^2E_N^2(t)
+C_5(\ve+\ve_1)^2\int_0^t\w{s}^{-1}E_N^2(s)ds,\\
&E_{N-9}(t)+\cX_{N-9}(t)\le C_5\ve+C_5\ve_1(\ve+\ve_1)^2,\\
&\w{t+|x|}^{0.9}|(Z^au)_n(t,x)|\le C_5|n|^{|a|+5-N}(\ve+\ve_1^3),\quad|a|\le N-7, n\in\Z_*.
\end{split}
\end{equation}
As the first step, by choosing $\ve^1_0=\frac{1}{2\sqrt{2C_5}}$, when $\ve, \ve_1\leq \ve^1_0$, we see from the first inequality
in \eqref{proof-of-theorem-cubic} that
\begin{equation*}\label{proof-of-theorem-cubic-1}
E_N^2(t)\leq 2C_5\ve^2+2C_5(\ve+\ve_1)^2\int_0^t\w{s}^{-1}E_N^2(s)ds.
\end{equation*}
By Gronwall's inequality and Lemma \ref{lem:aux:energy}, there are constants $C_6, C_7\geq 1$ such that
\begin{equation}\label{proof-of-theorem-cubic-1}
E_N(t)+\cX_N(t)\le C_6\ve(1+t)^{C_7(\ve+\ve_1)^2}+C_6\ve_1^3.
\end{equation}
Choosing $\ve_0=\min\{\frac{1}{8\sqrt{C_6}(C_5+C_6)},\frac{1}{2\sqrt{C_5}(1+4(C_5+C_6))},
\frac{1}{\sqrt{10C_7}(1+4(C_5+C_6))}\}$, $\ve_1=\mathrm{min}\{\ve_0^1,4(C_5+C_6)\ve\}$, $\ve_2=C_7(1+4(C_5+C_6))^2\ve^2 \le1/10$.
Then for $\ve\le\ve_0$ we can obtain
\begin{equation*}
\begin{split}
E_N(t)+\cX_N(t)&\le\frac12\ve_1(1+t)^{\ve_2},\\
E_{N-9}(t)+\cX_{N-9}(t)&\le\frac12\ve_1,\\
\w{t+|x|}^{0.9}|(Z^au)_n(t,x)|&\le\frac12\ve_1|n|^{|a|+5-N},\quad|a|\le N-7, n\in\Z_*.
\end{split}
\end{equation*}
This, together with the local existence of classical solution to \eqref{QWE-cubic} yields that \eqref{QWE-cubic}
with \eqref{null:def-cubic} admits a unique global solution $u\in \bigcap\limits_{j=0}^{N+1}C^{j}([0,\infty), H^{N+1-j}(\R^2\times\T)))$. Moreover,
\eqref{dyu:pw-cubic} can be achieved by \eqref{pw:zero'-cubic}, \eqref{pw:nonzero-cubic} and \eqref{dyu:pw'-cubic}.
\end{proof}

\begin{proof}[Proof of $(B_2)$ in Theorem \ref{thm1-cubic}]
In view of the third inequality in \eqref{proof-of-theorem-cubic} and \eqref{proof-of-theorem-cubic-1}~(once $\ve, \ve_1\leq \ve^1_0$), there
are constants $C_8,C_9\ge1~(C_8:=C_7, C_9:=\mathrm{max}\{C_5,C_6\})$ such that
\begin{equation*}
\begin{split}
E_N(t)+\cX_N(t)&\le C_9\ve(1+t)^{C_8(\ve+\ve_1)^2}+C_9\ve_1^3,\\
\w{t+|x|}^{0.9}|(Z^au)_n(t,x)|&\le C_9|n|^{|a|+5-N}(\ve+\ve_1^3),\quad|a|\le N-7, n\in\Z_*.
\end{split}
\end{equation*}
Let $\kappa_0=\frac{1}{C_8(1+4eC_9)^2}$, $\ve_0=\frac{1}{8e^2C_9^{3/2}}$, $\ve_1=4eC_9\ve$, $\ve_2=0$.
Then for $\ve\le\ve_0$ and $t\le T_\ve=e^{\kappa_0/\ve^2}-1$, we arrive at
\begin{equation*}
\begin{split}
E_N(t)+\cX_N(t)&\le\frac12\ve_1,\\
\w{t+|x|}^{0.9}|(Z^au)_n(t,x)|&\le\frac12\ve_1|n|^{|a|+5-N},\quad|a|\le N-7, n\in\Z_*.
\end{split}
\end{equation*}
This, together with the local existence of classical solution to \eqref{QWE-cubic}, ensures that \eqref{QWE-cubic}
admits a unique solution $u\in \bigcap\limits_{j=0}^{N+1}C^{j}([0,T_\ve], H^{N+1-j}(\R^2\times\T)))$.
In addition, \eqref{pw:zero'-cubic}, \eqref{pw:nonzero-cubic} and \eqref{dyu:pw'-cubic} lead to \eqref{dyu:pw-cubic}.
\end{proof}








\end{document}